\documentclass[]{amsart}

\allowdisplaybreaks

\usepackage[english]{babel}
\usepackage[utf8]{inputenc}

\usepackage{amsfonts}
\usepackage{amssymb}
\usepackage{amsthm}
\usepackage{amscd}
\usepackage{amsxtra}
\usepackage[all]{xy}
\usepackage{todonotes}
\usepackage{comment}
\usepackage{tikz}
\usepackage{float}

\usepackage{ upgreek }

\usepackage{enumerate}
\usepackage{dsfont}
\usepackage{epic}
\usepackage[dvips]{rotating}

\usepackage{todonotes}
\usepackage[pdftitle={},
pdfauthor={Matteo Costantini},
linktocpage=true
]{hyperref}

\usetikzlibrary{intersections,calc,arrows,backgrounds,decorations.markings}

\DeclareMathOperator{\Teich}{\mathcal{T}}        
\DeclareMathOperator{\HH}{\mathbb{H}}            
\DeclareMathOperator{\PP}{\mathbb{P}}            

\DeclareMathOperator{\Proj}{\mathcal{P}}    
\DeclareMathOperator{\Teichm}{\mathcal{T}} 
 
\DeclareMathOperator{\qfuchs}{\mathcal{QF}} 

\DeclareMathOperator{\CC}{\mathbb{C}}		  
\DeclareMathOperator{\RR}{\mathbb{R}}		  
\DeclareMathOperator{\NN}{\mathbb{N}}		  
\DeclareMathOperator{\QQ}{\mathbb{Q}}		  

\DeclareMathOperator{\err}{\mathrm{Err}}  
\DeclareMathOperator{\ev}{\mathrm{ev}}  

\DeclareMathOperator{\Tol}{\mathrm{Tol}}  
\DeclareMathOperator{\GL}{\mathrm{GL}}  
\DeclareMathOperator{\SU}{\mathrm{SU}}  
\DeclareMathOperator{\PU}{\mathrm{PU}}           
\DeclareMathOperator{\U}{\mathrm{U}}           
\DeclareMathOperator{\Aut}{\mathrm{Aut}}         
\DeclareMathOperator{\Hom}{\mathrm{Hom}}         
\DeclareMathOperator{\PSL}{\mathrm{PSL}}         
\DeclareMathOperator{\SL}{\mathrm{SL}}           
\DeclareMathOperator{\Sp}{\mathrm{Sp}}           
\DeclareMathOperator{\SO}{\mathrm{SO}}           
\DeclareMathOperator{\OO}{\mathrm{O}}           
\DeclareMathOperator{\dev}{\mathrm{dev}}   
\DeclareMathOperator{\hol}{\mathrm{hol}} 
\DeclareMathOperator{\tbad}{\mathrm{T_{bad}}}
\DeclareMathOperator{\tbadL}{\mathrm{T^{\mathcal{L}}_{bad}}}
\DeclareMathOperator{\tbadE}{\mathrm{T^{\mathcal{E}}_{bad}}}
\DeclareMathOperator{\grass}{\mathrm{Grass}} 
 
\DeclareMathOperator{\Op}{\mathrm{Op}} 
\DeclareMathOperator{\gr}{\mathrm{gr}}
\DeclareMathOperator{\pardeg}{{\mathrm{deg}}_{\mathrm{par}}}  
\DeclareMathOperator{\betti}{\mathcal{M}_{\mathrm{B}}^{(n)}}  
\DeclareMathOperator{\drham}{\mathcal{M}_{\mathrm{DR}}^{(n)}(C)}  
\DeclareMathOperator{\higgs}{\mathcal{M}_{\mathrm{H}}^{(n)}(C)}  
\DeclareMathOperator{\bettitwo}{\mathcal{M}_{\mathrm{B}}^{(2)}}  
\DeclareMathOperator{\drhamtwo}{\mathcal{M}_{\mathrm{DR}}^{(2)}(C)}  
  
\newcommand{\dd}{\mathrm{d}}			 	

\renewcommand{\Im}{\mathrm{Im\,}}       

\DeclareMathOperator{\vol}{\mathrm{vol}}        
\DeclareMathOperator{\rk}{\mathrm{rk}}         

\DeclareMathOperator{\hyp}{\mathrm{hyp}}         
\DeclareMathOperator{\Ad}{\mathrm{Ad}}           
 
\DeclareMathOperator{\Gr}{\mathrm{Gr}} 
\DeclareMathOperator{\Symm}{\mathrm{Sym}} 
\DeclareMathOperator{\dhyp}{\mathrm{d}_{\mathrm{hyp}}} 

\DeclareMathOperator{\Hcoh}{\mathrm{H}}   

\excludecomment{draftcomment}

\theoremstyle{plain}
\newtheorem{prop}{Proposition}[section]
\newtheorem{thm}[prop]{Theorem}
\newtheorem{lem}[prop]{Lemma}
\newtheorem{cor}[prop]{Corollary}

\newtheorem{theorem}{Theorem}

\theoremstyle{definition}
\newtheorem{defn}[prop]{Definition}

\theoremstyle{remark}
\newtheorem{rem}[prop]{Remark}

\newtheorem*{conjecture}{Conjecture}
\newtheorem{thmdef}[prop]{Theorem/Definition}

\numberwithin{equation}{section}

\title[Lyapunov exponents and  de Rham moduli space ]{Lyapunov exponents, holomorphic flat bundles and de Rham moduli space}

\author[M. Costantini]{Matteo Costantini}

\begin{document}
	
	\begin{abstract}
		We consider Lyapunov exponents for flat bundles over hyperbolic curves defined  via parallel transport over the geodesic flow. We refine a lower bound obtained by Eskin, Kontsevich, Möller and Zorich showing that the sum of the first $k$ exponents is greater or equal than the sum of the degree of any rank $k$ holomorphic subbundle of the flat bundle and the asymptotic degree of its equivariant developing map. We also show that this inequality is an equality if the base curve is compact. We moreover relate the asymptotic degree to the dynamical degree defined by Daniel and Deroin. 
		
		We then use the previous results to study properties of Lyapunov exponents on variations of Hodge structures and on Shatz strata of the de Rham moduli space. In particular we show that the top Lyapunov exponent function is unbounded on the maximal Shatz stratum, the oper locus. In the final part of the work we specialize to the rank two case, generalizing a result of Deroin and Dujardin about Lyapunov exponents of holonomies of projective structures. 
	\end{abstract}
	
	\maketitle

\section{Introduction}

Lyapunov exponents are characteristic numbers describing the behaviour of a cocycle over a dynamical system.  If the cocycle satisfies an integrability property, Oseledets theorem states that there is a decomposition of the underlying vector bundle such that the norm of vectors in each component grows with different speed along the flow. The different possible growth rates are called Lyapunov exponents.

An interesting instance of a dynamical system is given by playing billiard on tables of polygonal shape with angles that are rational multiples of $\pi$. Lyapunov exponents describe the diffusion rate of the trajectories of the ball. Even in this special case, Lyapunov exponents are very hard to compute using standard ergodic theoretic tools. There are two remarkable facts that allowed to get a hold onto these invariants. The first is that  the Lyapunov exponents given as diffusion rates of trajectories on a billiard given by a flat surface $(X,\omega)$ are the same as the ones of a completely different dynamical system, where the Lyapunov exponents are defined as the asymptotic growth rate of the Hodge norm of vectors  in  the variation of Hodge structures over the  flow in the affine invariant manifold $\overline{\SL_2(\RR)(X,\omega)}$. The second key tool used for computing Lyapunov exponents makes use of algebraic geometry. It is in \cite{ekz} where Eskin, Kontsevich and Zorich proved that the sum of positive Lyapunov exponents of the Kontsevich--Zorich cocycle over an affine invariant manifold can be computed by computing the normalized degree of the Hodge bundle restricted to the affine invariant manifold.  
Starting from billiards, algebraic geometry was used to  investigate Lyapunov exponents in more general settings. For example,  Kappes and Möller in \cite{martinandre} proved that a result analogous to the one of \cite{ekz} was true for weight one variations of Hodge structures over ball quotients. This result allowed  them to prove some results about commensurability questions for lattices.   Later, Filip \cite{simionk3} proved a similar result in the case of  variations of Hodge structures given by one dimensional families of K3 surfaces.
Variations of Hodge structures are a special case of flat vector bundles that are characterized by the existence of a special filtration and a compatible Hodge norm. However, Lyapunov exponents can be defined for a general flat vector bundle. Indeed,  a flat vector bundle defines a cocycle given by parallel transport over the geodesic flow.  In \cite{ekmz} Eskin, Kontsevich, M\"oller and Zorich proved that for a general flat vector bundle with non-expanding monodromy at the cusps, the sum of the first $k$ Lyapunov exponents is bounded from below by the normalized degree of any  rank $k$ holomorphic subbundle of the flat bundle. An analogous result was obtained by Daniel and Deroin in \cite{deroindaniel} in which they considered Lyapunov exponents obtained from Brownian motion over K\"ahler manifolds.

In this paper, we firstly refine the main inequality of  \cite{ekmz} by relating the sum of the first $k$ exponents to the sum of the normalized degree of a holomorphic subbundle and the asymptotic degree of its developing map. This is the first main  Theorem \ref{thm:lyapexp}.

\begin{theorem}
	Let  $\mathcal{V}$ be a  holomorphic flat bundle over a  finite area hyperbolic Riemann surface $C$ with non-expanding cusp monodromy. Let $\overline{C}$ be the completion of $C$ and $\Delta=\overline{C}\setminus C$ be the set of cusps. 
	For any  holomorphic subbundle   $\mathcal{E}\subset \mathcal{V} $ of rank $k$, it holds
	\begin{equation*}
		\sum_{i=1}^k \lambda_i\geq  \frac{2\pardeg(\Xi_h(\mathcal{E}))}{\deg(\Omega_{\overline{C}}^1(\log(\Delta))}  + \err^{\mathcal{E}}(u) 
	\end{equation*}
	for  almost all $x\in C$ and   Lebesgue almost all $u\in \bigwedge^{k} \mathcal{V}^{\vee}_x $. The error term is defined as
	\[ \err^{\mathcal{E}}(u)=\pi\lim_{T\to \infty}  \frac{1}{T}\int_{0}^{T}\frac{\sharp \{s_{\mathcal{E}}^{-1}(\ker u)\cap D_t\}}{\vol(D_t)} \dd t.\]
	Here we denoted by $\pardeg(\Xi_h(\mathcal{E}))$  the parabolic degree of the metric extension of $\mathcal{E}$ by some admissible metric $h$, by $s_{\mathcal{E}}:\HH\to \mathbb{P}(\bigwedge^{k} \mathcal{V}_x )$  the holomorphic classifying map defining $\mathcal{E}$  and by $D_t$  the hyperbolic ball of radius $t$ in the hyperbolic plane whose center is a lift of $x$. 
\end{theorem}

The proof of the above theorem  generalizes to variations of Hodge structures of weight one over ball quotients or to the canonical variation of Hodge structures of weight one over affine invariant submanifolds of the Hodge bundle.

We say that  the  holomorphic flat bundle $\mathcal{V}$ is $k$-irreducible if the holomorphic flat bundle $\bigwedge^k \mathcal{V}$ is irreducible.
If the base curve is compact and if $\mathcal{V}$ is $k$-irreducible, we show in the other main  Theorem \ref{thm:conjcompact} that the sum of the first $k$ exponents is actually equal to the right hand side of the above inequality.

\begin{theorem}
	If the Riemann surface $C$ is compact and the  holomorphic flat bundle $\mathcal{V}$ is $k$-irreducible, the above inequality  is an equality.
\end{theorem}

The above results can be compared to the main result of \cite{deroindaniel}, in which Daniel and Deroin get a similar equality in the context of the Brownian motion flow over a K\"ahler manifold. They prove that  the sum of the first  Lyapunov exponents is the same as  the sum of a normalized degree and a dynamical degree.  As a corollary we get that in the case of compact base curves the error term $\err^{\mathcal{E}}(u)$ defined above, which has the same shape as the covering degree defined in \cite{dd1} in the context of Lyapunov exponents for holonomies of projective structures, is the same as the dynamical degree defined in \cite{deroindaniel}.

In the second part of the paper we consider Lyapunov exponents as invariants on the de Rham moduli space of flat vector bundles over a compact hyperbolic Riemann surface $C$. We first concentrate on the special subset of this moduli space given by variations of Hodge structures. We show a simplicity result for positive weight variations of Hodge structures and prove slightly generalized versions of the result of \cite{ekz} and \cite{simionk3} about weight one and weight two variations of Hodge structures. 
We then concentrate on other loci given by Shatz strata, defined by Harder--Narasimhan type. In particular, via the identification of the maximal stratum with the set of opers, in Proposition \ref{prop:lyapoper} and Theorem \ref{thm:shatzunbounded} we show the following.

\begin{theorem}
	If $\mathcal{V}$ is in the maximal Shatz stratum of the space of rank $n$ flat vector bundles, the oper locus, then 
	\[\sum_{i=1}^k \lambda_i(\mathcal{V})\geq k(n-k),\quad k=1,\dots,n.\]
	The above inequalities are sharp, since they are achieved in the special point corresponding to the only flat vector bundle of the maximal Shatz stratum  underlying a variation of Hodge structures, which also corresponds to the  $(n-1)$-th symmetric power of the uniformizing representation of $C$.
	
	Moreover, the top Lyapunov exponent function is unbounded on this  stratum and the function has logarithmic growth near its boundary.
\end{theorem}

We finally focus on the study of the Lyapunov exponent functions in rank two. We describe special loci and some properties of the Lyapunov exponent functions. Moreover we show that the oper locus in rank two is the same as the set of holonomies of projective structures with the same underlying Riemann surface structures. With this point of view we retrieve the main results of \cite{deroin} from a different point of view. 
 
Finally with the support of computer experiments we state a conjecture which is a statement analogous to the last theorem for Hitchin components. In this setting would be interesting to relate Lyapunov exponents to other invariants like the critical exponent or the minimal area function.

\subsection{Structure of the paper}

In Section \ref{sec:deflyap} we state Oseledets multiplicative ergodic theorem and  define Lyapunov exponents for flat vector bundles. We set the notation for parabolic vector bundles and parabolic degree and state the main result of \cite{ekmz}.

In Section \ref{sec:mainresultlyap} we prove a refinement of the previous result of \cite{ekmz} showing in Theorem \ref{thm:lyapexp} that the  sum of the first $k$ exponents is greater or equal than the sum of the degree of any rank $k$ holomorphic subbundle of the flat bundle and the asymptotic degree of its equivariant developing map. We also state an analogous version of the theorem in the more general context where the base space is a ball quotient or an affine invariant manifold of abelian differentials. We finally give necessary conditions to ensure rationality of the sum of Lyapunov exponents.

In Section \ref{sec:equality} we present Theorem \ref{thm:conjcompact}, which states that  if the base curve is compact  the main inequality is actually an equality. This allows us to relate the asymptotic degree of our formula in the compact case  to the dynamical degree defined in \cite{deroindaniel}. 

In Section \ref{sec:vhsrat} we then focus on other special flat bundles, namely the ones defining a variation of Hodge structures. Using  the above cited  general criterion for ensuring the rationality of the sum of exponents, we reprove slightly more general versions of the original results of \cite{ekz} and \cite{simionk3} about weight one and real weight two variations of Hodge structures. For a general variation of Hodge structures $\mathcal{V}$, we prove  that  the Lyapunov spectrum is trivial if and only if  $\mathcal{V}$ is unitary and we show that if the weight is positive we get a non-trivial bound on the sum of the first $\rk(\mathcal{V}^{n,0})$-exponents. 

In Section \ref{subsec:R-H} we consider Lyapunov exponents as functions on the De Rham moduli space and investigate the Shatz stratification of this moduli space. We then focus on the maximal stratum, which we identify with the oper locus, providing lower bounds for the Lyapunov exponent functions and showing the unboundedness of these functions.

 In Section \ref{sec:rank2} we focus on the moduli space of rank two flat bundles, giving a summarizing picture of important subsets of this moduli space. In particular we show that the oper locus is the same as the set of holonomies of projective structures underlying the same Riemann surface structure.  
 
In Section \ref{sec:lyapexpfctrank2} we describe how the Lyapunov exponent function behaves on these loci, proving  the main results of \cite{deroin} from our point of view. We also give alternative proofs of the known continuity of the Lyapunov exponent function and of the locus where this function is zero.

Finally,  in Section \ref{sec:unbounded} we conjecture,  with the support of computer experiments, that the top Lyapunov exponent function on the Hitchin components is unbouded from above and bounded from below  by the value of the top Lyapunov exponent of the appropriate symmetric power of the uniformizing representation of $C$. The lower bound would be a result analogous to the one of \cite{poitriesambarino} about the critical exponent.

\subsection*{Acknowledgements.}
We thank Martin M\"oller for suggesting this project to us and for many fruitful and valuable discussions. 
We  thank Simion Filip and Jeremy Daniels for important remarks about weight two variations of Hodge structures and Bertrand Deroin for the insights about branching projective structures.

\section{Lyapunov exponents for parabolic flat bundles}
\label{sec:deflyap}

In this section we set the notation and recall the definition of Lyapunov exponents, focusing on the setting of parallel transport for parabolic flat vector bundles over the geodesic flow of a hyperbolic Riemann surface. We finally recall the main result of \cite{ekmz} that will be refined in the next section.

\subsection{Oseledets theorem}

We want to define the Lyapunov exponents associated to a flat vector bundle over a complete hyperbolic Riemann surface $C$ of finite area.

We first of all  recall Oseledets  multiplicative ergodic theorem and the definition of Lyapunov exponents for a cocycle over an ergodic flow.

\begin{thmdef}
	Let $g_t:(M,\mu)\to(M,\mu)$ be an ergodic flow on a space $M$ with finite measure $\mu$. Suppose that the action lifts equivariantly to a linear flow $G_t$ on some measurable real bundle $\mathcal{V}$ on $M$. Suppose there exists a (not equivariant) norm $\|\cdot\|$ on $\mathcal{V}$ such that the functions
	\begin{equation}
		x\mapsto \sup_{t\in [0,1]} \log^+\|G_t\|_x,\quad x\mapsto \sup_{t\in [0,1]} \log^+\|G_{1-t}\|_{g_t(x)}
	\end{equation}
	are in $L^1(M,\mu)$ (we call such a norm \textit{integrable}). Then there exist real constants $\lambda_1\geq \dots \geq \lambda_n$ and a decomposition 
	\[\mathcal{V}=\bigoplus_{i=1}^n \mathcal{V}_{\lambda_i}\]
	by measurable real vector bundles such that for a.a. $x\in M$ and all $v\in (\mathcal{V}_{\lambda_i})_x- \{0\}$, it holds
	\[\lambda_i=\lim_{t\to\pm\infty} \frac{1}{t}\log\|G_t(v)\|.\] 	
	The set of values $\lambda_i$, repeated with multiplicity $\dim \mathcal{V}_{\lambda_i}$, is called the set of \textit{Lyapunov exponents} or Lyapunov spectrum of $(M,\mu,g_t,\mathcal{V})$.
\end{thmdef}

The case we are interested in is when the space $M$ is the unit tangent bundle $T^1(C)$ of a hyperbolic curve $C\cong \HH/\Gamma$ equipped with the measure induced by the hyperbolic metric with constant negative curvature $-4$ (we keep the same convention as in \cite{ekmz}). 
We denote by $\Delta:=\overline{C}\setminus C$ the boundary points with respect to the smooth compactification of $C$. 
Let $\rho:\Gamma\to \GL_n(\CC)$ be a representation of the fundamental group.
We denote by $\mathbb{V}_{\rho}$ the local system on $C$ induced by $\rho$ and by $\mathcal{V}_C:=\mathbb{V}_{\rho}\otimes_{\CC} \mathcal{O}_C$ the associated holomorphic bundle equipped with the flat Gauss-Manin connection $\nabla$.
The geodesic flow $g_t:T^1(C)\to T^1(C)$ is  ergodic and we can lift it to the flat vector bundle $(\mathcal{V}_C,\nabla)$ using parallel transport. We will consider Lyapunov exponents with respect to this situation.

\begin{rem}
	\label{rem:symmspec}
	The Lyapunov spectrum is symmetric in this case. This follows since the reverse time geodesic flow, which reverses the signs of the Lyapunov exponents, is conjugate to the positive time flow (because of the $\SL_2(\RR)$-action on $\HH$).  Moreover, if the vector bundle $\mathcal{V}$ is complex, it is possible to show that the Oseledets  subvector bundles are complex subbundles of $\mathcal{V}$. In this case we then consider only half of the real Lyapunov spectrum, forgetting about the duplication phenomen given by the complex structure.
\end{rem}

We still need to define an integrable norm on the flat bundle $(\mathcal{V}_C,\nabla)$. We can consider the \textit{constant norm} given as the pullback to $T^1(C)$ of the parallel transport of any norm over the fiber of  some base point $c\in C$ to a Dirichlet fundamental domain for $\Gamma$ on $\HH$. In particular this norm is not continuous across the boundary of the fundamental domain.
We say that $(\mathcal{V}_C,\nabla)$ has  non-expanding cusp monodromies if the eigenvalues of the holonomy matrices  $\text{hol}_{\nabla}(\gamma)$ have absolute value one, for every simple loop $\gamma$ around a cusp.
We recall  a result which ensures us that  the constant norm is integrable.
\begin{thm}[{\cite[Lem. 2.6, Sec. 2.7]{ekmz}}]
	\label{thm:integrable}
	The constant norm over $(V,\nabla)$ is integrable if and only if the associated local system has non-expanding cusp monodromies.
\end{thm}

Note that for any two  integrable norms on $(\mathcal{V}_C,\nabla)$, the Lyapunov filtrations and the  Lyapunov spectra coincide (see \cite[Lemma 2.6]{martinandre}).

\begin{rem}
	The constant norm can be used to numerically compute Lyapunov exponents associated to $(V,\nabla)$ when a coding for the geodesic flow on $\HH/\Gamma$ is available.
	If we denote by $\gamma_n\in \Gamma$ the sequence of elements corresponding to the sequence of sides of the fundamental domain crossed by a generic geodesic, then by definition of constant norm the Lyapunov exponents are given as
	\[\lambda_i=\lim_{n\to \infty}\frac{1}{n}\log(\mu_i(n)),\] 
	where $\mu_i(n)$ are the eigenvalues of the matrix  
	\[ \text{hol}_{\nabla}(\gamma_n)\cdot\dots\cdot \text{hol}_{\nabla}(\gamma_1).\]	
\end{rem}

\subsection{Parabolic flat bundles}

We now summarize results and notions of \cite{simvhs}, \cite{sim90} and \cite{ekmz}. We will introduce parabolic bundles and metric extensions, which are needed in order to state the main result of \cite{ekmz} in  the case of non compact base curves.

\begin{defn}
	A \textit{parabolic bundle} $\mathcal{V}$ over $\overline{C}$ is a holomorphic vector bundle together with a $[0,1)$-filtration $F^{\cdot}\mathcal{V}_c$ on  the fiber $\mathcal{V}_c$
	\[\mathcal{V}_c=\mathcal{V}_c^{\geq \alpha_1}\supsetneq \mathcal{V}_c^{\geq \alpha_2}\supsetneq\dots \supsetneq \mathcal{V}_c^{\geq \alpha_{n+1}}=\mathcal{V}_c^{\geq 1}=0\]
	for every $c\in \Delta$.
\end{defn}
A morphism between parabolic bundles $\mathcal{V}$ and $\mathcal{E}$ is a morphism of
holomorphic vector bundles such that for each $c \in \Delta$  and each weight $\alpha$ of $\mathcal{V}_c$ the image of $\mathcal{V}^{\geq\alpha}$ lies in $\mathcal{E}^{\geq\beta}$ whenever $\beta\leq \alpha$. A \emph{parabolic subbundle} $i:\mathcal{E}\hookrightarrow\mathcal{V}$ is an injective morphism of parabolic bundle with the additional requirements that for each $c\in \Delta$ the weights of $\mathcal{E}$ are a subset of the weights of $\mathcal{V}$ and if $\beta$ is maximal such that
$i(\mathcal{E}^{\geq\alpha})\subseteq \mathcal{V}^{\geq \beta}$ then $\alpha=\beta$.

If we denote by $0\leq \alpha_1<\alpha_2<\dots <\alpha_n<\alpha_{n+1}=1$ the weights of the filtration of a fiber $\mathcal{V}_c$, the filtered dimension of  $(\mathcal{V}_c,F^{\cdot})$ is defined as
\[\dim_{F^{\cdot}}(\mathcal{V}_c)=\sum_{i=1}^n \alpha_i\dim \gr_{\alpha_i}(\mathcal{V}_c)\]
where $\gr_{\alpha_i}(\mathcal{V}_c)$ is the graded piece at weight $\alpha_i$.

\begin{defn}
	The \textit{parabolic degree} of a parabolic bundle $(\mathcal{V},F^{\cdot})$ is defined to be
	\[\pardeg(\mathcal{V},F^{\cdot})=\deg(\mathcal{V})+\sum_{c\in \Delta} \dim_{F^{\cdot}}\mathcal{V}_c.\]
\end{defn}

Following \cite{ekmz}, we  define  acceptable metrics. This notion is useful in order to compute parabolic degrees of parabolic bundles.

\begin{defn}
	A smooth metric $h$ on a holomorphic vector bundle $\mathcal{V}_C$ over $C$ is called \emph{acceptable} if the curvature $R_h$ of the metric admits locally near every cusp $c\in \Delta$, a bound
	\[|R_h|\leq f+\frac{M}{|q|^2|\log(q)|^2}\]
	with $f\in L^p(C)$ for some $p>1$ and some constant $M$.	
\end{defn}

When we consider a holomorphic vector bundle $\mathcal{V}_C$ over $C$ equipped with a smooth metric $h$, we can talk about a canonical metric extension of $\mathcal{V}_C$ on $\overline{C}$, which in general is just a coherent sheaf. 

Let $j:C\hookrightarrow \overline{C}$ be the inclusion.

\begin{defn}
	The \emph{metric extension} $\Xi_h(\mathcal{V}_C)$  of $\mathcal{V}_C$   to $\overline{C}$ with respect to the metric $h$ is given by the subsheaf  of $j_*(\mathcal{V}_C)$  defined by following growth condition. If  $s(q)$ is a local section of $j_*(\mathcal{V}_C)$ over $U\subset \overline{C}$ around a cusp $c\in U\cap \Delta$,   we  set $s(q)\in \Xi_h(\mathcal{V}_C)(U)$ if for all $\epsilon> 0$ there exists a constant $M({\epsilon})$ such that $|s(q)|_{h}\leq M({\epsilon})|q|^{-\epsilon}$.    
	
	The parabolic structure of $\Xi_h(\mathcal{V}_C)$ over the cusps is given by the following filtration. If  $s(q)$ is a local section of $\Xi_h(\mathcal{V}_C)$ around a cusp $c\in \Delta$,   we  set $s\in (\Xi_h(\mathcal{V}_C))_c^{\geq \alpha}$ if for all $\epsilon> 0$ there exists a constant $M({\epsilon})$ such that \[|s(q)|_{h}\leq M({\epsilon})|q|^{\alpha-\epsilon}.\]
\end{defn}

We want now to generalize the definition of acceptable metric to parabolic bundles over $\overline{C}$.

\begin{defn}
	A smooth metric $h$ on a parabolic vector bundle $(\mathcal{V},F^{\cdot})$ over $\overline{C}$ is called \emph{acceptable} if $h$ is an acceptable metric for the holomorphic bundle $\mathcal{V}_{|C}$ and $(\mathcal{V},F^{\cdot})=\Xi_h(\mathcal{V})$.
\end{defn}

We recall now a result which allows us to compute the parabolic degree of a parabolic vector bundle using any acceptable metric.

\begin{prop}[{\cite[Prop. 2.5]{ekmz}}]
	\label{lem:pardeg}
	If $(\mathcal{V},F^{\cdot})$ is a parabolic vector bundle over $\overline{C}$ of rank $k$, then 
	\[\pardeg (\mathcal{V})=\pardeg(\wedge^k \mathcal{V}).\]
	Moreover if $h$ is an acceptable metric, then
	\[\pardeg(\mathcal{V},F^{\cdot})=\frac{1}{2\pi i}\int_C \partial \overline{\partial}\log(\det(h_{ij}))\]
	where $h_{ij}=h(e_i,e_j)$ is the Gram matrix of the metric.
\end{prop}

Finally we can define the notion of admissible metric on a holomorphic flat bundle over $C$. Admissible metrics are the ones that can be used to compute Lyapunov exponents.

For any flat holomorphic vector bundle over $C$ there is a canonical extension,  which is called the Deligne extension. 
It is a holomorphic vector bundle $\mathcal{V}$ on $\overline{C}$ with a logarithmic connection $\nabla:\mathcal{V} \to \mathcal{V}_{\rho}\otimes \Omega^1_{\overline{C}}(\log(\Delta))$.  Note  that for holomorphic flat bundles $(\mathcal{V}_C,\nabla)$ over $C$ with non-expanding cusp monodromies the Deligne extension has a canonical parabolic structure (see \cite{ekmz}).

\begin{defn}
	A smooth metric $h$ on the holomorphic flat bundle $\mathcal{V}_C$ over $C$ is called \emph{admissible} if  the following conditions hold:
	\begin{enumerate}
		\item The metric $h$ is acceptable for the Deligne extension $\mathcal{V}$ of $\mathcal{V}_C$ with respect to its canonical parabolic structure.
		\item For every cusp $c\in \Delta$ with coordinate $q$, there is some $n\in \NN$ such that for any $ e\in \mathcal{V}_c^{\geq \alpha}$ and $ e'\in \mathcal{V}_c^{\geq \alpha'}$  it holds
		\[  h(e(q),e'(q))\leq M_1|q|^{\alpha+\alpha'}(\log|q|)^{2n},\quad \text{for some $M_1>0$}. \]
		\item For every cusp $c\in \Delta$ with coordinate $q$, there is some $n\in \NN$  such that a local generating section $e$ of $\det(\mathcal{V})$ has the lower bound
		\[|e(q)|_h\geq M_2|q|^{2\dim_{F^{\cdot}}(\mathcal{V}_c)} (\log|q|)^{-2n} ,\quad \text{for some $M_2>0$}. \]
	\end{enumerate}
\end{defn}

We want to highlight that the first condition of the above definitions simply says that  for an admissible metric the Deligne extension is the same as the metric extension and that the curvature does not grow too fast around the cusps.

We recall the existence lemma for such metrics  based on the result by Simspon \cite[Theorem 4]{sim90}.

\begin{lem}[{\cite[Lem. 2.4]{ekmz}}]
	\label{lem:existadmissible}
	If a holomorphic flat bundle $(\mathcal{V}_C,\nabla)$ has non-expanding cusp monodromies, then it admits an admissible metric.
\end{lem}

We say that a norm on $\mathcal{V}_C$ is admissible if it is induced by an admissible metric.
Recall that if $\mathcal{V}_C$ has non-expanding cusp monodromies, by Theorem \ref{thm:integrable} the constant norm is integrable in the sense of Oseledets theorem. In this case the same property holds for admissible norms. 

\begin{lem}[{\cite[Lem. 2.7]{ekmz}}]
	\label{lem:integrableadmissible}
	If a holomorphic flat bundle $(\mathcal{V}_C,\nabla)$ has non-expanding cusp monodromies, then any admissible norm is integrable in the sense of Oseledets theorem.
\end{lem}

Admissible norms can then be used to compute Lyapunov exponents.

\subsection{Relation of Lyapunov exponents and parabolic degrees of subbundles}

We recall the main result of the paper by Eskin, Konstevich,Möller and Zorich. It relates the sum of the first $k$-Lyapunov exponents to the degree of a rank $k$ holomorphic subbundle.

\begin{thm}[\cite{ekmz}]
	\label{thm:ekmz}
	Let $(\mathcal{V}_C,\nabla)$ be a holomorphic flat bundle with non-expanding cusp monodromies. If $\mathcal{E}\subset \mathcal{V} $ is a holomorphic parabolic subbundle of rank $k$ of the Deligne extension $\mathcal{V}$ of $\mathcal{V}_C$, then 
	\begin{equation}
	\label{eq:bound}
	\sum_{i=1}^k \lambda_i(\mathcal{V}_C)\geq \frac{2 \pardeg(\mathcal{E})}{\deg(\Omega_{\overline{C}}^1(\log(\Delta))}.
	\end{equation}
	
\end{thm}

\section{Refinement of \cite{ekmz}-inequality and higher dimensional analogues}
\label{sec:mainresultlyap}

In this section we prove one of the main results, namely a refinement of the main theorem of \cite{ekmz}. 
In order to state the refinement of inequality \eqref{eq:bound}, we need to recall the definition of a seminorm on $\bigwedge^k \mathcal{V}_C^{\vee}$ introduced in \cite{ekmz} in the proof of Theorem \ref{thm:ekmz}.

Let $h$ be an admissible metric for $\mathcal{V}_C$.  By abuse of notation, we will denote also by $h$ the induced admissible metric on  dual exterior powers of $\mathcal{V}_C$. We denote by $|\cdot|_{h}$ the induced norm. Following \cite{ekmz}, we define a seminorm on $\bigwedge^k \mathcal{V}_C^{\vee}$. For every point $c\in C$, consider a small open set $c\in U_c\subset C$ and fix  a local basis $\omega_1,\dots,\omega_k$ of $H^0(U_c,\mathcal{E})$. We define locally the seminorm on $\bigwedge^k\mathcal{V}_C^{\vee}$ as
\begin{equation}
\label{eq: seminorm}
\|u\|_{\mathcal{E}}=\frac{|u(\omega_1(c)\wedge \dots \wedge \omega_k(c))|}{\|\omega_1(c)\wedge \dots \wedge \omega_k(c)\|_h},\quad   u\in\bigwedge^{k} \mathcal{V}^{\vee}_c\cong \left(\bigwedge^{k} \mathcal{V}_c\right)^{\vee}   . 
\end{equation}

The seminorm does not depend on the choice of the local frame $(\omega_i)$ of $\mathcal{E}$ since numerator and denominator are homogenous of the same degree in $\omega_i$'s.

The subvector bundle defined by the zero locus of the seminorm  will appear in the refined inequality that we want to prove. 
In this regard, let us pull-back the vector bundles to the universal covering $\pi:\HH\to C $ of $C$.
Since any local system on a simply connected space is trivial,  the pull-back of the flat bundle  $\mathcal{V}_C$ is trivial. Let us fix $c\in C$ and a lift $\tilde{c}\in \HH$. Via the canonical isomorphism $\mathcal{V}_c\cong \pi^*(\mathcal{V}_C)_{\tilde{c}}$ we fix the trivialization 
\[\pi^*(\mathcal{V}_C)\cong \HH\times \mathcal{V}_c.\]
The pull-back $\pi^*(\mathcal{E}_{|C})\subset \pi^*(\mathcal{V}_C)$ defines a $\rho_{\mathcal{V}_C}$-equivariant subbundle, where $\rho_{\mathcal{V}_C}$ is the representation associated to the flat bundle $\mathcal{V}_C$. This means that
\begin{equation}
\label{eq:equiv}
\pi^*(\mathcal{E}_{|C})_{\gamma(\tilde{c})}=\{\gamma(\tilde{c})\}\times \rho_{\mathcal{V}_C}(\gamma)(\mathcal{E}_c)\subset \HH\times \mathcal{V}_c
\end{equation}
for every $\gamma\in  \pi_1(C,c)$.

\begin{defn}
	\label{def:badlocus}
	We define the 'trouble making set', or bad locus,   
	\[\tbadE:=\ker\left(\bigwedge^k \mathcal{V}_C^{\vee}\to \bigwedge^k \mathcal{E}_{|C}^{\vee}\right)\subseteq \bigwedge^k \mathcal{V}_C^{\vee}\]
	as the kernel of the map induced by the inclusion $\mathcal{E}\subseteq \mathcal{V}$.\\
	Moreover, for any $u\in   \bigwedge^k \mathcal{V}_c^{\vee}=H^0(\HH,\pi^*(\bigwedge^k \mathcal{V}_C^{\vee}))$, the   'trouble making set' associated to $u$, or bad locus of $u$, is defined as
	\[\tbadE(u):=\{z\in \HH\colon \|u_z\|_{\pi^*(\mathcal{E})}=0\}\subseteq \HH.\]
\end{defn}

Notice that $\tbadE\subseteq \bigwedge^k \mathcal{V}_C^{\vee}$ is a holomorphic subbundle of corank $1$ and $\tbadE(u)$ can be identified with the intersection of the pull-back $\pi^*(\tbadE)$ to $\HH$ with the horizontal foliation induced by the flat bundle 
$\bigwedge^k \mathcal{V}_C^{\vee}$.
\begin{rem}
	\label{rem:geombad}
	
	To give an inclusion of a rank $k$ holomorphic sub-vector bundle $ \mathcal{E}_{|C}\subset \mathcal{V}_C $ over $C$ is equivalent to give a  section 
	\[s_{\mathcal{E}}:C\to \mathcal{G}_k(\mathcal{V}_C)\]
	of the Grassmanian bundle $\mathcal{G}_k(\mathcal{V}_C)$ of $k$-planes of $\mathcal{V}_C$.
	Via the above choice of trivialization $\pi^*(\mathcal{V}_C)\cong \HH\times \mathcal{V}_c$,
	 the pull-back section $\pi^*(s_{\mathcal{E}})$ defines a $\rho_{\mathcal{V}_C}$-equivariant holomorphic map
	\[s_{\mathcal{E}}:\HH\to \grass(k,\mathcal{V}_c)\hookrightarrow \PP(\bigwedge^k \mathcal{V}_c),\quad s_{\mathcal{E}}(z)=(\pi^*(\mathcal{E})_z\subset \mathcal{V}_c).\]
	The equivariance property is defined by the equality 
	\[s_{\mathcal{E}}(\gamma\cdot z)=\rho_{\mathcal{V}_C}(\gamma)\cdot s_{\mathcal{E}}(z),\quad \forall \gamma\in \pi_1(C,c).\]
	For any $u\in   \bigwedge^k \mathcal{V}_c^{\vee}\cong \left(\bigwedge^k \mathcal{V}_c\right)^{\vee}$, we can now rewrite its bad locus  as 
	\begin{equation}
	\label{eq:badlocusgeom}
	\tbadE(u)=\{z\in \HH\colon s_{\mathcal{E}}(z)\in \ker(u)\subset \PP(\bigwedge^k \mathcal{V}_c)\}.
	\end{equation}
	From this description it is clear that either $\tbad(u)=\HH$ if $\ker(u)\supseteq \Im(s_{\mathcal{E}})$ or it is a discrete subset given as the zero set of an holomorphic (non $\rho_{\mathcal{V}_C}$-equivariant) function on $\HH$.
\end{rem}

We observe that if the representation $\rho_{\mathcal{V}_C}$ satisfies an irreducibility property, then we can ensure that $\tbad(u)\not =\HH$ for every $u\in  \bigwedge^k \mathcal{V}_c^{\vee}$. We will need to use that $\tbad(u)$ is discrete in $\HH$ for every $u\in  \bigwedge^k \mathcal{V}_c^{\vee}$ for the proof of Theorem \ref{thm:conjcompact}.

\begin{lem}
	\label{lem:badlocusnottrivial}
If $\rho_{\mathcal{V}_C}$ is $k$-irreducible, meaning that $\bigwedge^k \rho_{\mathcal{V}_C}$ is irreducible, then $\tbad(u)\not =\HH$ for every $u\in  \bigwedge^k \mathcal{V}_c^{\vee}$.
\end{lem}
\begin{proof}
Using the point of view of Remark \ref*{rem:geombad}, we have that $\tbad(u)=\HH$ if and only if $\Im(s_{\mathcal{E}})\subset \ker(u)$. This implies that the linear span of $\Im(s_{\mathcal{E}})$ is contained in the hyperplane $\ker(u)$. Since $s_{\mathcal{E}}$ is $\bigwedge^k \rho_{\mathcal{V}_C}$-equivariant,  its linear span is a $\bigwedge^k \rho_{\mathcal{V}_C}$-invariant subspace of $\bigwedge^k \mathcal{V}_c$. This gives a contradiction, since by assumption $\bigwedge^k \rho_{\mathcal{V}_C}$ is irreducible.
\end{proof}

Notice that parallel transport on the trivial bundle given by the pull-back of $\mathcal{V}^{\vee}$ to $T^1\HH$ is simply given, after the choice of a trivialization, by the constant transport
\[G_t:T^1\HH\times\mathcal{V}^{\vee}_c\to T^1\HH\times\mathcal{V}^{\vee}_c,\quad G_t(x,u)=(g_t(x),u).\]
From now on we fix a choice of a trivialization and we denote a point in the pull-back bundle by $u_z:=(z,u)\in T^1\HH\times \mathcal{V}^{\vee}_c$  and a lift of $c\in C$ by $\tilde{c}\in \HH$ .

We will state now a refinement of the main Theorem of \cite{ekmz}.
Let  the notation be as above. In particular $\mathcal{V}_C$ is a  holomorphic flat bundle over $C$ defined by a representation $\rho_{\mathcal{V}_C}$. 
\begin{thm}
	\label{thm:lyapexp}
	For any  holomorphic subbundle   $\mathcal{E}\subset \mathcal{V}_C $ of rank $k$ over $C$, then 
	\begin{equation}
	\label{eq:bounderr}
	\sum_{i=1}^k \lambda_i(\mathcal{V}_C)\geq  \frac{2\pardeg(\Xi_h(\mathcal{E}))}{\deg(\Omega_{\overline{C}}^1(\log(\Delta))}  + \pi\lim_{T\to \infty}  \frac{1}{T}\int_{0}^{T}\frac{\sharp \{\tbad(u)\cap D_t(\tilde{c})\}}{\vol(D_t(\tilde{c}))} \dd t
	\end{equation}
	for  almost any $c\in C$ and   Lebesgue almost any $u\in \bigwedge^{k} \mathcal{V}^{\vee}_c $.  Here $D_t(\tilde{c})$ denotes the hyperbolic ball of radius $t$ in $\HH$ with center $\tilde{c}$. 
\end{thm}


\begin{proof}	
	
	First of all  note that it suffices to prove the theorem in the case where $\mathcal{E}$ is a line bundle. Indeed if it is not the case, consider the line bundle $\mathcal{L}:=\bigwedge^k \mathcal{E}\subset \bigwedge^k \mathcal{V}_C$. Then $\pardeg(\mathcal{L})=\pardeg(\mathcal{E})$ by Lemma \ref{lem:pardeg} and the top Lyapunov exponent of $\mathcal{L}$ is just the sum $\sum_{i=1}^k \lambda_i(\mathcal{V}_C)$ of the first $k$ exponents.
	Hence from now on $\mathcal{E}=\mathcal{L}$ is a sub-line bundle of $\mathcal{V}_C$.
	Moreover, since the Lyapunov spectrum is symmetric (see Remark \ref{rem:symmspec}), the Lyapunov spectrum of the dual local system $\mathbb{V}^{\vee}$ is the same as the one of $\mathbb{V}$. We will then focus on computing the top Lyapunov exponent $\lambda_1(\mathcal{V}_C^{\vee})=\lambda_1(\mathcal{V}_C)$.
	
	Note that the  Cauchy-Schwartz inequality implies that the admissible norm $||\cdot||_h$ induced by $h$ is greater or equal than the $\mathcal{L}$-seminorm $||\cdot||_{\mathcal{L}}$ defined in \eqref{eq: seminorm}. Indeed, for any $c\in C$ and any $u\in \mathcal{V}_c^{\vee}$ it holds 
	\begin{equation}
	\label{eq:cauchy-scwhartz}
	\|u\|_{\mathcal{L}}=\frac{|u(\omega_c)|}{\|\omega_c\|_h}\leq \frac{||u||_h\||\omega_c||_h}{\|\omega_c\|_h}=||u||_h
	\end{equation}
	where $\omega$ is a local non-zero section of $\mathcal{L}$ near $c\in C$.\\
	By Lemma \ref{lem:integrableadmissible}, the norm induced by the admissible metric $h$ is integrable, meaning that it computes the Lyapunov exponents. 
	We can then write the top Lyapunov exponent as
	\[\lambda_1(\mathcal{V}_C)=\lim_{t\to \infty} \frac{1}{t}(\log||G_t(u)||_h)\]
	for almost any $(c,v)\in T^1(C)$ and Lebesgue almost any $u\in T^1\mathcal{V}_{(c,v)}^{\vee}$. Here we denoted by $T^1\mathcal{V}^{\vee}$ the pull-back of $\mathcal{V}^{\vee}$ to $T^1(C)$.\\
	We apply the usual chain of equalities as in \cite{ekz} or \cite{ekmz} to rewrite the expression above. We first average over the circle  and then use the above Cauchy–Schwarz inequality \eqref{eq:cauchy-scwhartz}. After that we take the integral of the derivative and then we apply a version of Green's formula (\cite[Lemma 3.6]{ekz}) for the hyperbolic disc $D_t(\tilde{c})$ centered in $\tilde{c}\in \HH$ with hyperbolic radius $t$ (here the term $\log\|G_t r_{\theta} u\|_{\mathcal{L}}$ is considered in the distributional sense). Finally we split the integral using the definition of the $||\cdot||_{\mathcal{L}}$-seminorm and rewrite directly  the second term of the expression in terms of the degree of $\mathcal{L}$ as in \cite{ekmz}, where the main ingredient is given by equidistribution of discs of large radius which allows to replace the limit of integrals over balls with an integral over $C$.
	\begin{align*}
		\lambda_1(\mathcal{V}_C)&=\lim_{T\to \infty} \frac{1}{T}\frac{1}{2\pi}\int_{0}^{2\pi} \log\|G_T r_{\theta} u\|_h\dd\theta\geq \lim_{T\to \infty} \frac{1}{T}\frac{1}{2\pi}\int_{0}^{2\pi} \log\|G_T r_{\theta} u\|_{\mathcal{L}}\dd\theta\\
		&= \lim_{T\to \infty} \frac{1}{T}\frac{1}{2\pi}\int_0^T \frac{\dd}{\dd t}\int_{0}^{2\pi} \log\|G_T r_{\theta} u\|_{\mathcal{L}}\dd\theta\\
		&= \lim_{T\to \infty} \frac{1}{T}\int_{0}^{T} \frac{\tanh (t)}{2\vol(D_t(\tilde{c}))} \int_{D_t(\tilde{c})} \Delta_{\hyp} \log\|u_z\|_{\mathcal{L}}\dd g_{\hyp}(z)\ \dd t\\
		&= \lim_{T\to \infty} \frac{1}{T}\int_{0}^{T} \frac{\tanh (t)}{2\vol(D_t(\tilde{c}))} \int_{D_t(\tilde{c})} \Delta_{\hyp} \log|u_z(\omega_z)|\dd g_{\hyp}(z)\ \dd t\ +\\
		&- \lim_{T\to \infty} \frac{1}{T}\int_{0}^{T} \frac{\tanh (t)}{2\vol(D_t(\tilde{c}))} \int_{D_t(\tilde{c})} \Delta_{\hyp} \log\|\omega_z\|_{h}\dd g_{\hyp}(z)\ \dd t\\
		&=\lim_{T\to \infty} \frac{1}{T}\int_{0}^{T} \frac{\tanh (t)}{2\vol(D_t(\tilde{c}))} \int_{D_t(\tilde{c})} \Delta_{\hyp} \log|u_z(\omega_z)|\dd g_{\hyp}(z)\ \dd t\ +\\
		&+\frac{2\deg_{par}(\Xi_h(\mathcal{L}))}{\deg(\Omega_{\overline{C}}^1(\log(\Delta))}
	\end{align*}
	Note that we could split the $\log$ only in the fourth line, since  the Laplacian makes the numerator and the denominator of the $\mathcal{L}$-norm become well-defined functions.\\
	We need to treat the first summand. We write explicitly the hyperbolic Laplacian and the hyperbolic area form and simplify. We then use that the integral over the ball of the distribution $\overline{\partial}\partial \log(|u_z(\omega_z)|)$ gives  the number of zeros of  the holomorphic function $u_z(\omega_z)$ inside the ball times $\pi/ i$ (cf. \cite[Poincaré-Lelong Equation]{gh}). The last equality follows since $\tanh(t)$ is bounded and asymptotic to $1$ for large $t$.
	
	\begin{equation*}
		\begin{split}	
			&\lim_{T\to \infty} \frac{1}{T}\int_{0}^{T} \frac{\tanh (t)}{2\vol(D_t(\tilde{c}))} \int_{D_t(\tilde{c})} \Delta_{\hyp} \log|u_z(\omega_z)|\dd g_{\hyp}(z)\ \dd t\\
			&=\lim_{T\to \infty} \frac{1}{T}\int_{0}^{T} \frac{\tanh (t)}{2\vol(D_t(\tilde{c}))} \left(\int_{D_t(\tilde{c})} 4\frac{\partial^2}{\partial z\partial \overline{z}} \log|u_z(\omega_z)| \frac{i}{2}|\dd z|^2\right) \dd t\\ 	
			&= i\lim_{T\to \infty} \frac{1}{T}\int_{0}^{T} \frac{\tanh (t)}{\vol(D_t(\tilde{c}))} \left(\int_{D_t(\tilde{c})}  \overline{\partial}\partial  \log|u_z(\omega_z)|\right) \dd t\\
			&= \pi  \lim_{T\to \infty} \frac{1}{T}\int_{0}^{T} \frac{\tanh (t)}{\vol(D_t(\tilde{c}))}\sharp\{ z\in D_t(\tilde{c})\colon u_z(\omega_z)=0\}  \dd t\\
			&= \pi  \lim_{T\to \infty} \frac{1}{T}\int_{0}^{T} \frac{\sharp\{ z\in D_t(\tilde{c})\colon u_z(\omega_z)=0\} }{\vol(D_t(\tilde{c}))} \dd t
		\end{split}
	\end{equation*}
	
\end{proof}

\begin{defn}
	We define the second term in formula \ref{eq:bounderr} as the \emph{error term}
	\[\err^{\mathcal{E}}(u):=\pi\lim_{T\to \infty}  \frac{1}{T}\int_{0}^{T}\frac{\sharp \{\tbadE(u)\cap D_t(\tilde{c})\}}{\vol(D_t(\tilde{c}))} \dd t\]
	for $u\in \bigwedge^k \mathcal{V}_c^{\vee}$.
\end{defn}

\begin{rem}
	The error term $\err^{\mathcal{E}}(u)$ is defined as the limit of the mean of a normalized counting function. This limit exists since we showed it is the difference between the Lyapunov exponents and the parabolic degree. Notice that if not only the limit of the mean of the counting function, but also the limit of the counting function itself $\displaystyle \lim_{t\to \infty} \frac{\sharp \{\tbadE(u)\cap D_t(\tilde{c})\}}{\vol(D_t(\tilde{c}))}$ exists, then the error term is equal to this limit. We conjecture that this is the case.  Note moreover that $\err^{\mathcal{E}}(u)=\err^{\mathcal{E}}(\lambda u)$, for any $\lambda\in \CC^*$. Hence the error term defines a function
	\[\err^{\mathcal{E}}:\PP(\bigwedge^k \mathcal{V}_C^{\vee}) )\to \RR^+. \]
\end{rem}

\subsection{Higher dimensional analogues}

Using the same argument as in the proof of Theorem \ref{thm:lyapexp}, one can prove analogous statements in the case where the base manifold is a ball quotient or an affine invariant manifold of a stratum of abelian differentials. In the first case, since ball quotients are locally symmetric spaces of rank 1, the geodesic flow is ergodic and so the Oseledets  multiplicative ergodic theorem can be applied. In the second case there is a natural $\SL_2(\RR)$ ergodic action on affine invariant manifolds. We omit the details of the proofs since the computations are analogous to the the ones of the last theorem.

The next Proposition is a generalization of the main result of \cite{martinandre}.

\begin{prop}
	Let $\mathcal{V}$ be a weight one variation of Hodge structures over a ball quotient $B=\mathbb{B}^n/\Gamma$ of constant curvature $-4$, where $\Gamma$ is a torsionfree lattice in $\PU(1,n)$. Let $\overline{B} $ be a smooth compactification of $B$ with normal crossing boundary divisor $\Delta$. Let $\mathcal{E}\subset \mathcal{V}$ be a holomorphic sub-vector bundle of rank $k$. Then
	\begin{align*}
		\sum_{i=1}^k \lambda_i&\geq \frac{(n+1)c_1(\Xi_h(\mathcal{E}))\cdot c_1(\omega_{\overline{B}})^{n-1}}{c_1(\omega_{\overline{B}})^{n}}+\\
		& \quad +\frac{2i}{2n(n-1)!}\lim_{T\to \infty}\frac{1}{T}\int_0^T \frac{1}{\vol(\mathbb{B}^n_t)} \int_{\mathbb{B}^n_t(\tilde{c})}\partial\overline{\partial}(\log|u(s_z)|)\wedge \omega_{\text{hyp}}^{n-1} \dd g_{\text{hyp}}(z)\dd t
	\end{align*}
	for almost any $c\in B$ and Lebesgue almost any vector  $u\in \bigwedge^k \mathcal{V}^{\vee}_c$. Here   $s$ is a local generator of $\bigwedge^k \mathcal{E}$ and $\mathbb{B}^n_t(\tilde{c})$ is the hyperbolic ball of radius $t$ around the lift $\tilde{c}\in \mathbb{B}^n$ of $c\in B$. Finally $\omega_{\overline{B}}=\bigwedge^n \Omega^1_{\overline{B}}(\log (\Delta))$ is the log-canonical bundle and $\Xi_h(\mathcal{E})$ is the metric extension of $\mathcal{E}$ with respect to the Hodge metric $h$.
\end{prop}

In the last proposition we only considered weight one variation of Hodge structures instead of general flat bundles  since the integrability of the Hodge norm was proven in this case in \cite{martinandre} using the geometry of the period domain and Royden's theorem. The statements of Lemma \ref{lem:existadmissible} and Lemma \ref{lem:integrableadmissible} should extend to the case of ball quotients, and in this case the above result can be generalized to any flat vector bundle.

The next proposition is about the case of affine invariant manifolds. It is a generalization of \cite{ekz}.

\begin{prop}
	Let $\mathcal{M}_1$	be an affine invariant manifold in some stratum of abelian differentials. Let $\mathcal{H}$ be the Hodge bundle and $\mathcal{E}\subset \mathcal{H}$ be a holomorphic sub-vector bundle of rank $k$. Then the sum of the top $k$ Lyapunov exponents associated to the ergodic probability measure $\nu_1$ corresponding to $\mathcal{M}_1$ satisfy the bound:
	\[\sum_{i=1}^k \lambda_i\geq \int_{\mathcal{M}_1}\Delta(\log||\omega||_h)\dd \nu_1+\err^{\mathcal{E}}(u)\]
	where the Laplacian is the leafwise Laplacian along Teichmueller disks  and $||\omega||_h$ is the Hodge norm of a local section $\omega$ of $\bigwedge^k \mathcal{E}$.  Finally the error term is considered along  any Teichmueller disk passing through the base point of $u$, for almost any $c\in \mathcal{M}_1$ and Lebesgue almost any $u\in \bigwedge^k \mathcal{H}_c^{\vee}$.
\end{prop}

Notice that the error term is considered along \textit{any} Teichmueller disk, since the Oseledets theorem holds for every Teichmueller disk by \cite{chaikaeskin}. Notice moreover that the error term in the last proposition  depends \textit{only} on the Teichmueller disk passing through the base point of $u$ and the restriction of $\mathcal{E}$ to this Teichmueller disk. 

\subsection{Condition for rationality of Lyapunov exponents}

We want now to state a sufficient condition for the sum of the top Lyapunov exponents being equal to the first term of inequality \eqref{eq:bounderr}. This gives in particular a sufficient condition for the sum of the top Lyapunov exponents to be rational.

\begin{prop}
	\label{prop:conditionemptylocus}
	Let  $\mathcal{S}\subset\PP (\bigwedge^{k} \mathcal{V}^{\vee}) $  be a $G_t$-invariant closed subset such that there is a vector $u\in \mathcal{S}$  computing the top Lyapunov exponents, namely such that
	$\sum_{i=1}^k\lambda_i=\lim_{t\to \infty} \frac{1}{t}(\log||G_t(u)||_h)$. If there is a rank $k$ holomorphic subbundle $\mathcal{E}\subset \mathcal{V}$ such that \[\tbadE(u)=\emptyset,\quad \text{for all } u\in \mathcal{S}\] then
	\[\sum_{i=0}^k \lambda_i=  \frac{2\pardeg(\Xi_h(\mathcal{E}))}{\deg(\Omega_{\overline{C}}^1(\log(\Delta))}.\]
\end{prop}

\begin{rem}
	
	This last proposition is the analogous of \cite[Prop. 3.15]{deroindaniel} in which they require a strong irreducility property of the flat bundle in order to have the right harmonic measure.  We  do not need any irreducibility property, but the drawback is that we need the existence of a vector computing the sum of the top exponents. In \cite{deroindaniel} they do not need this assumption since for any closed $G_t$-invariant $\mathcal{S}$ there is always a harmonic measure with support in $\mathcal{S}$. 
\end{rem}

Note that if the main inequality  \eqref{eq:bounderr} of Theorem \ref{thm:lyapexp} were an equality, we would not need the existence of the additional subbundle $\mathcal{S}$ but only the existence of a vector $u\in\bigwedge^{k} \mathcal{V}^{\vee}$ computing the top Lyapunov exponents with $\tbadE(u)=\emptyset$. Indeed in this case the error term $\err^{\mathcal{E}}(u)$ would be zero and this would suffice.  Since we will prove that over compact base curve \eqref{eq:bounderr} is an equality (Theorem \ref{thm:conjcompact}), we can apply the previous argument to this situation (Corollary \ref{cor:conditionemptycompact}). 

\begin{rem}
	Consider the Grassmanian bundle $\Gr(n-k,\mathcal{V})$ of $(n-k)$-planes in $\mathcal{V} $ as a subset of $\PP (\bigwedge^{k} \mathcal{V}^{\vee})$ via the Plucker embedding $\Gr(n-k,\mathcal{V})\subset \PP (\bigwedge^{n-k} \mathcal{V}) $ and the  isomorphism $\PP (\bigwedge^{k} \mathcal{V}^{\vee})\cong \PP (\bigwedge^{n-k} \mathcal{V})$. Then the condition that a $(n-k)$-plane $u\in \Gr(n-k,\mathcal{V})$ has emty bad locus, i.e. $ \tbadE(u)=\emptyset$, is equivalent to the condition that the $(n-k)$-plane in $\mathcal{V}$ represented by $u$ intersects trivially the $k$-plane defined by the subbundle $\mathcal{E}$.
	We will indeed use this criterion  to reprove rationality of Lyapunov exponents for weight 1 and K3 variation of Hodge structures in Section \ref{sec:vhsrat}. The $(n-k)$-plane computing the top Lyapunov exponents will be constructed from the Oseledets  subspaces $\mathcal{V}_{\lambda_i}$.
\end{rem}

\begin{proof}[Proof of Proposition \ref{prop:conditionemptylocus}]
	
	We want to prove that if the bad locus is empty for all $u\in \mathcal{S}$, then we can use the $\mathcal{E}$-norm to compute Lyapunov exponents. \\
	The argument is the standard one relying on the equivalence of any two norms on a finite dimensional vector space.
	If the bad locus $\tbadE(u)=\emptyset$ is empty for all $u\in \mathcal{S}$, it means that the $\mathcal{E}$-norm $||\cdot||_{\mathcal{E}}$ is a norm on the  $G_t$-invariant closed subset  $ \mathcal{S}  \subset \PP(\bigwedge^{k} \mathcal{V}^{\vee})$. Let $K\subset T^1C$ be a compact positive measure set. Then  $\mathcal{S}_{|K}\subset\PP(\bigwedge^{k} \mathcal{V}^{\vee}) $ is a compact subset and the quotient of the norms $||\cdot||_{\mathcal{E}}$ and $||\cdot||_h$ defines a bounded function on $\mathcal{S}_{|K}$ with minimum greater than zero. This means that  there exist two positive constants $C_1$ and $C_2$ such that 
	\[C_1||u||_h\leq ||u||_{\mathcal{E}}\leq C_2||u||_h,\quad \forall u\in \mathcal{S}_{K}.\]
	Now by Poincaré recurrence Theorem, the geodesic flow on $T^1C$ comes back infinitely many times to $K$ since it has positive measure. Moreover $\mathcal{S}$ is $G_t$-invariant by assumption. Let $t_j$  be a sequence of times tending to infinity for which $g_{t_j}(c)\in K$. Now let $u\in \mathcal{S}$ be the vector computing  the top Lyapunov exponents, which exists by assumption. We get then
	\begin{align*}
		\sum_{i=1}^k\lambda_i&=\lim_{t\to \infty} \frac{1}{t}(\log||G_t(u)||_h)=\lim_{t_j\to \infty} \frac{1}{t_j}(C_1\log||G_{t_j}(u)||_h)\leq \\
		\leq   &\lim_{t_j\to \infty} \frac{1}{t_j}(\log||G_{t_j}(u)||_{\mathcal{E }})\leq \lim_{t_j\to \infty} \frac{1}{t_j}(C_2\log||G_{t_j}(u)||_h)=\\
		&=\lim_{t\to \infty} \frac{1}{t}(\log||G_t(u)||_h)=\sum_{i=1}^k\lambda_i.
	\end{align*}
	
	By following the proof of Theorem \ref{thm:lyapexp}, we see that if the error term is computed with respect to the vector $ u$ used above the inequality \eqref{eq:bounderr} becomes an equality. The claim then follows directly from it. 
\end{proof}

The condition to have empty bad locus $\tbadE(u)$ for $u\in \bigwedge^{k} \mathcal{V}^{\vee}$ can be rephrased using the equivalent definition of $\tbadE(u)$ given by expression \eqref{eq:badlocusgeom} via the equivariant map $s_{\mathcal{E}}:\HH\to \PP(\bigwedge^{k} \mathcal{V}_c)$ defining $\mathcal{E}$:
\begin{equation}
\label{eq:conditionequality}
\tbadE(u)=\emptyset\quad  \text{ if and only if }\quad \Im(s_{\mathcal{E}})\cap \ker(u)=\emptyset \subset  \PP(\bigwedge^{k} \mathcal{V}_c).
\end{equation}

\begin{rem}
	\label{rem:rank2condition}
	If the vector bundle $\mathcal{V}$ is of rank $2$, then an element $u\in \bigwedge^{1} \mathcal{V}^{\vee}_c$ defines a point in $\PP(\mathcal{V}_c)$. If $\mathcal{E}\subset \mathcal{V}$ is a sub-line bundle, the developing map defining $\mathcal{E}$
	\[s_{\mathcal{E}}:\HH\to \PP^1_{\CC}\]
	is simply a meromorphic function on $\HH$ equivariant with respect to the action of the representation $\rho_{\mathcal{V}}$ defined by $\mathcal{V}$.  Then 
	\[\tbadE(u)=\{z\in \HH\colon s_{\mathcal{E}}(z)= u \in \PP^1_{\CC}\}.\]
	Moreover in this case there is only one line $u\in \PP(\mathcal{V}_c)$ not computing the top Lyapunov exponent, namely the line corresponding to the second Oseledets  space $\mathcal{V}_{\lambda_2}$. 
\end{rem} 
The previous remark together with Proposition \ref{prop:conditionemptylocus} imply the following condition for equality in the rank 2 situation .
\begin{cor}
	\label{cor:equalityrank2noncompact}
	Let $\mathcal{V}$ be a rank 2 flat bundle over a hyperbolic Riemann surface $C$ and $\mathcal{E}\subset \mathcal{V}$ a sub-line bundle. If there is a $\rho_{\mathcal{V}}$-invariant subset  $S\subset \PP^1_{\CC}$ containing more than one point and  such that $s_{\mathcal{E}}(\HH)\cap S=\emptyset$, then 
	\[ \lambda_1=  \frac{2\pardeg(\Xi_h(\mathcal{E}))}{\deg(\Omega_{\overline{C}}^1(\log(\Delta))}.\]
\end{cor}
\begin{proof}
	The pull-back of the restriction of the tautological bundle $\mathcal{O}_{\PP_{\CC}^1}(-1)_{|S}$ defines a $G_t$-invariant closed subset $\mathcal{S}\subset \PP(\mathcal{V})$. Since $S$ contains more than two points, by the previous remark there is at least one line $u\in S$ computing the top Lyapunov exponent. By Proposition \ref{prop:conditionemptylocus} we then have equality.
\end{proof}
An important example of an invariant closed subset  $S\subset \PP_{\CC}^1 $ containing more than one point is the limit set of a discrete faithful representation. We can then specialize the last corollary in this setting.

\begin{cor}
	\label{cor:kleininalimit}
	Let $\mathcal{V}$ be a rank 2 flat bundle over a hyperbolic Riemann surface $C$ corresponding to a faithful discrete representation $\rho_{\mathcal{V}}$. If there is a $\rho_{\mathcal{V}}$-equivariant holomorphic map $\dev:\HH\to \PP_{\CC}^1$ whose image is disjoint to the limit set of $\rho_{\mathcal{V}}$, then 
	\[ \lambda_1=  \frac{2\pardeg(\Xi_h(\dev^{*}(\mathcal{O}_{\PP_{\CC}^1}(-1))))}{\deg(\Omega_{\overline{C}}^1(\log(\Delta))}.\]
	Here $\dev^{*}(\mathcal{O}_{\PP_{\CC}^1}(-1))$ is an abuse of notation for the line bundle on $C$ defined by the map $\dev$.
\end{cor}

We will see in Section \ref{sec:projstr} that the data of a representation $\rho_{\mathcal{V}}$ together with an equivariant immersion $\dev:\HH\to \PP_{\CC}^1$ is equivalent to the datum of a projective structure on a surface, so the previous corollary can be applied to the setting of projective structures.

\section{Main equality in the compact case}

\label{sec:equality}

In this section we show that inequality \eqref{eq:bounderr} is an equality if the base curve is compact and $\mathcal{V}$ is $k$-irreducible.  Recall that a flat bundle is called  $k$-irreducible if its $k$-exterior power is irreducible.

\begin{thm}
	\label{thm:conjcompact}
	Let $\mathcal{V}$ be a  $k$-irreducible flat bundle over a compact hyperbolic Riemann surface $C$. For every holomorphic subbundle $\mathcal{E}\subset \mathcal{V}$, it holds
	\[	\sum_{i=0}^k \lambda_i(\mathcal{V}_C)=  \frac{2\deg(\mathcal{E})}{\deg(\mathcal{K}_C)} +\err^{\mathcal{E}}(u) \]
	for almost any $c\in C$ and Lebesgue almost any $\displaystyle u\in \bigwedge^k \mathcal{V}_c^{\vee}$.
\end{thm}

We will first state some applications of the previous result and then go on with its proof. The proof of Theorem \ref{thm:conjcompact} is quite technical and is based on finer estimates on the bad locus. We finally recall the main result of \cite{deroindaniel} and get as a corollary that the dynamical degree defined in \cite{deroindaniel} is the same as our  error term if the base curve is compact.

\subsection{Applications}

Thanks to the equality proven in Theorem \ref{thm:conjcompact} we get a better condition in the case of compact base curves for checking rationality of Lyapunov exponents than  the one given by Proposition \ref{prop:conditionemptylocus}. 

\begin{cor}
	\label{cor:conditionemptycompact}
	Let $\mathcal{V}$ be a  $k$-irreducible flat bundle over a compact hyperbolic Riemann surface $C$.
	If there is a rank $k$ holomorphic subbundle $\mathcal{E}\subset \mathcal{V}$ such that \[\tbadE(u)=\emptyset \] for a vector $u\in \bigwedge^k \mathcal{V}^{\vee}$ that computes the sum of the top Lyapunov exponents, then
	\[\sum_{i=0}^k \lambda_i=  \frac{2\deg(\mathcal{E})}{\deg(\mathcal{K}_C)}.\]
\end{cor}

Note that the previous corollary can be used for example if one considers the vector  $u=\sum_{i=k}^{n}\mathcal{V}_{\lambda_i}\in \bigwedge^k \mathcal{V}^{\vee}$ given by the sum of the last Oseledets  spaces. Then only the emptiness condition $\tbadE(u)=\emptyset$ has to be checked.

In the rank 2 situation we get a better version of Corollary \ref{cor:equalityrank2noncompact} in the case of compact base curve. In this case we know indeed that there is only one line $\mathcal{V}_{\lambda_2}$ not computing the top Lyapunov exponent.

\begin{cor}
	\label{cor:equalityrank2compact}
	Let $\mathcal{V}$ be an irreducible rank 2 flat bundle over a compact hyperbolic Riemann surface $C$ and $\mathcal{E}\subset \mathcal{V}$ a sub-line bundle. If the complement $\PP_{\CC}^1\setminus s_{\mathcal{E}}(\HH)$ of the image  of the corresponding equivariant map $s_{\mathcal{E}}:\HH\to \PP_{\CC}^1$ contains more than one point, then 
	\[ \lambda_1= \frac{2\deg(\mathcal{E})}{\deg(\mathcal{K}_C)}.\]
\end{cor}

Since a rank $k$ holomorphic subbundle of the flat bundle corresponding to a representation $\rho$  is the same as a $\wedge^k \rho$-equivariant holomorphic map  $f:\HH\to \PP(\bigwedge^k \CC^n)$, we get also the following corollary.

\begin{cor}
	Let $C$ be a compact Riemann surface. For any $k$-irreducible representation of the fundamental group $\rho:\pi_1(C) \to \GL_n(\CC)$  and any $\wedge^k \rho$-equivariant holomorphic map $f:\HH\to \PP(\bigwedge^k \CC^n)$ the error term function 
	\[\PP(\bigwedge^k \CC^n)^{\vee}\to \RR^+,\quad	u\mapsto \err^{f}(u)=\lim_{T\to \infty}  \frac{1}{T}\int_{0}^{T}\frac{\sharp \{f^{-1}(\ker(u))\cap D_t(z)\}}{\vol(D_t(z))} \dd t \]
	is Lebesgue almost everywhere constant, for almost all $z\in \HH$. \\
\end{cor}

\subsection{Proof of Theorem \ref{thm:conjcompact} }
We will prove Theorem \ref{thm:conjcompact} by proving that  in the $k$-irreducible and compact base curve case, the $\mathcal{E}$-seminorm can be used to compute Lyapunov exponents as any other integrable norm. We want to remark that in order to prove  that the $\mathcal{E}$-seminorm is as good as any other integrable norm,  one could try  the naive approach via Poincaré recurrence theorem used for proving that any norm computes the same Lyapunov exponents (see for example \cite[Lemma 2.6]{martinandre}). The key idea of that approach is that any two norms are uniformly bounded with respect to each other on a projective bundle over a compact subset. In this case the bad locus breaks the compactness of the projective bundle since the $\mathcal{E}$-seminorm is a norm on the complement of the bad locus, which is not compact. 

In order to prove that the $\mathcal{E}$-seminorm can be used to compute Lyapunov exponents, it suffices as before to only consider  the case where $\mathcal{E}=\mathcal{L}$ is a line bundle.
Theorem \ref{thm:conjcompact} is a direct consequence of the following proposition, whose proof will take up the rest of this section.

\begin{prop}
	\label{prop:conjcompact}
	Let $C$ be compact. Let   $\mathcal{L}\subset \mathcal{V}$ be a holomorphic  subline bundle of an irreducible flat vector bundle $ \mathcal{V}$ over $C$. We denote by the same letters the pullbacks to the unit tangent bundle $T^1C$.
	For almost any $x\in T^1C$ and any vector  $u\in \mathcal{V}^{\vee}_x-\oplus_{i=2}^n \mathcal{V}_{\lambda_i,x}^{\vee}$ it holds
	\[\lambda_1(\mathcal{V})=\lim_{t\to \infty} \frac{1}{t}\frac{1}{2\pi}\int_{0}^{2\pi} \log\|G_t r_{\theta} u\|_{\mathcal{L}}\dd\theta\]
\end{prop}


From now on we will denote the trouble making sets introduced in Definition \ref{def:badlocus} as 
\[T:=\tbadL\subset \PP(\mathcal{V}^{\vee}),\quad T(u):=\tbadL(u)\subset \HH \text{ for any } u\in \PP(\mathcal{V}^{\vee}).\]

Let $\pi:\HH\to C$ be the universal covering map and $\pi^*(\PP(\mathcal{V})^{\vee})$ be the pullback of the projective bundle associated to $\mathcal{V}^{\vee}$. Since $\mathcal{V}^{\vee}$ is a flat bundle, the pull-back $ \pi^*(\PP(\mathcal{V}^{\vee}))$ is isomorphic to the trivial projective bundle. Let us fix an isomorphism \[\psi:\HH\times \PP(\mathcal{V}^{\vee}_c) \overset{\sim}{\longrightarrow}\pi^*(\PP(\mathcal{V}^{\vee})),\quad (z,u)\mapsto u_z:=\psi (z,u)\] 
for some $c\in C$.

Consider the function
\begin{align*}
	\phi: \HH\times \PP(\mathcal{V}^{\vee}_c)&\longrightarrow \RR_{\geq 0}\cup \{\infty\} \\
	(z,u)&\longmapsto  \log||u_z||_h-\log||u_z||_{\mathcal{L}}= \log\left(\frac{||u_z||_h }{||u_z||_{\mathcal{L}}}\right)=\log\left(\frac{||u_z||_h||\omega_z||_h }{|u_z(\omega_z)|}\right)
\end{align*}
where $\omega$ is a local frame of $\pi^*(\mathcal{L})$.
Notice that this function is positive by Cauchy-Schwartz (see inequality \eqref{eq:cauchy-scwhartz}).

We denote by $\phi_u:\HH\to \RR\cup \{\infty\}$ the map $\phi(-,u)$, for $u\in \PP(\mathcal{V}^{\vee}_c)$.

\begin{rem}
	Recall that the function appearing in the denominator of $\phi$ comes from the norm of the holomorphic function
	\begin{equation*}
		\label{eq:pairing}
		\HH\times \mathcal{V}^{\vee}_c\longrightarrow \CC,\quad (z,u)\mapsto u_z(\omega_z).
	\end{equation*}
	The bad locus $\psi^{-1}(T)\subset \HH\times \PP^{\vee}(\mathcal{V}_c)$ is its zero locus and the bad locus for the vector $u\in\PP(\mathcal{V}_c) $ is given by the slice $T(u)=\psi^{-1}(T)\cap \left(\HH\times \{u\}\right)$. Notice that by Lemma \ref{lem:badlocusnottrivial} and by the irreducibility hypothesis, the set $T(u)$ is discrete in $\HH$ for every $ u\in \PP(\mathcal{V}^{\vee})$.
\end{rem}

Let $\epsilon>0$ be a positive constant. We define a tubular neighborhood  of the bad locus $\psi^{-1}(T)\subset\HH\times \PP(\mathcal{V}^{\vee}_c) $ to be
\[B(T,\epsilon):=\{(z,u)\in \HH\times \PP(\mathcal{V}^{\vee}_c)\colon \dd_{\hyp}(z,T(u))<\epsilon\}\subset \HH\times \PP(\mathcal{V}^{\vee}_c)\]
and the slice 
\[B(T(u),\epsilon):=B(T,\epsilon)\cap \left(\HH\times\{u\}\right).\]
Let $B(T,\epsilon)^\complement\subset \HH\times \PP^{\vee}(\mathcal{V}_c)$ be the complement of the tubular neighborhood $B(T,\epsilon)$.
In the next lemma we obtain a bound on the behavior of the function $\phi$ on $B(T,\epsilon) $ and on $B(T,\epsilon)^\complement$. The main ingredients used in the proof of the next lemma are  the compactness of the curve and the equivariance property of $\phi$.

\begin{lem}
	\label{lem:tbad}	
	There exist constants $M,N>0$ such that the function $\phi$ outside the tubular neighborhood  $B(T,\epsilon)$ satisfies the following bound 
	\[||\phi_{|B(T,\epsilon)^{\complement}}||_{\infty}\leq M+N|\log(\epsilon)|.\]
	Moreover there is a constant $M'>0$ such that for every $u\in \PP(\mathcal{V}^{\vee}_c)$ and any $w\in T(u)$ the function $\phi_u$ restricted to the ball $B_{\hyp}(w,\epsilon)$ around $w$ satifies the following bound:
	\[\phi_u(z)_{|B_{\hyp}(w,\epsilon)}\leq M'+N\sum_{z'\in T(u)\cap B_{\hyp}(w,\epsilon)}|\log\left(\dhyp(z,z')\right)| .\]
\end{lem}

\begin{proof}
	
	Let us choose a compact fundamental domain $F\subset \HH$ for $C$  and consider the restrictions 
	\[T_{|F}:=\psi^{-1}(T)\cap  \left(F\times \PP(\mathcal{V}^{\vee}_c)\right)\subset \HH\times \PP(\mathcal{V}^{\vee}_c)\] and 
	\[B(T,\epsilon)_{|F}:=B(T,\epsilon)\cap  \left(F\times \PP(\mathcal{V}^{\vee}_c)\right).\]
	
	Note that if we prove the two claims of the proposition restricting ourselves to the subset $F\times \PP(\mathcal{V}^{\vee}_c)$, then we can use the equivariance property
	\[\phi(\gamma z,u)=\phi(z,\rho(\gamma^{-1})u),\quad \gamma\in \pi_1(C,c)\]
	to extend the results to all of the upper half plane since the constants involved in the expressions are independent of $u\in \PP(\mathcal{V}^{\vee}_c)$ and since $\pi_1(C,c)$ acts via isometries on $\HH$.
	
	The main idea now is considering the holomorphic function of expression \eqref{eq:pairing} locally as a power series. Then by a compactness argument we can control the coefficients of this power series. The main technical problem is that two zeros of this function, which are the bad points, can collide for some values of $u\in \PP(\mathcal{V}^{\vee})$. This has to be taken into account in order to correctly prove the second statement of the lemma.

	Since $F\times \PP(\mathcal{V}^{\vee}_c)$ is compact,  we can choose finitely many points $(z_j,u_j)\in T_{|F} $ and compact neighborhoods of the points $(z_j,u_j)\in U_j\subset F\times \PP(\mathcal{V}^{\vee}_c)$ such that $T_{|F}\subseteq \bigcup_{j=1}^m U_j$. We can now apply the  Weierstrass preparation theorem (see for example \cite[Ch.0.1]{gh}) to the holomorphic function $u_z(\omega_z)$ near the points $(z_j,u_j)$ and obtain 
	\[\phi(z,u)_{|U_j}=\log\left( \frac{||u_z||_h||\omega_z||_h }{\left|h_j(z,u)P_j(z,u)\right|}\right)\]
	where the holomorphic functions $h_j$ are never zero and the polynomials $P_j$ are given as
	\[P_j(z,u)=\sum_{i=0}^{n_j-1} a_{i,j}(u) (z-z_j)^i+(z-z_j)^{n_j}.\]
	The coefficients $a_{i,j}(u)$ are holomorphic functions  with $a_{i,j}(u_j)=0$ for all $i$ and $j$. Notice that 
	$T_{|F}\cap U_j$ is the zero locus of $P_j$.
	
	On each $U_j$ we get then the following bound:
	\begin{equation*}
		\begin{split}
			\phi(z,u)_{|U_j}&=\log\left( \frac{||u_z||_h||\omega_z||_h }{|h_j(z,u)|}\right)-\log|P_j(z,u)|\leq M''-\log|P_j(z,u)|
		\end{split}
	\end{equation*}
	where the constant
	\[M'':=\max_{j=1,\dots,m} \left(\max_{(z,u)\in U_j} \log\left( \frac{||u_z||_h||\omega_z||_h }{|h_j(z,u)|}\right)\right)\]
	is well defined since the functions $h_j$ are never zero on the compact subsets $U_j$.
	
	We then rewrite for every $(z,u)\in U_j$ the roots decomposition of the polynomial $P_j(z,u)$ with respect to the variable $z$ to get
	\begin{equation*}
		\begin{split}
			\phi(z,u)_{|U_j} &\leq M''-\log|P_j(z,u)|=M''-\sum_{i=1}^{n_j}\log|z-z_{i,j}(u)|
		\end{split}
	\end{equation*}
	where $z_{i,j}(u)\in \CC$ are possibly equal to each other. 
	Since $F\subset \HH$ is compact, the euclidean and the hyperbolic distances are comparable to each other. In particular there is a constant $L>0$ such that $|x-y|\geq L\cdot\dhyp(x,y)$ for all $x,y\in F$.	
	Hence we can rewrite the last inequality as
	\begin{equation}
	\label{eq:partialbound}
	\begin{split}
	\phi(z,u)_{|U_j} &\leq M''-\sum_{i=1}^{n_j}\log|z-z_{i,j}(u)|\leq M''-\sum_{i=1}^{n_j}\log|L\cdot \dhyp(z,z_{i,j}(u))|.
	\end{split}
	\end{equation}
	
	If $\epsilon$ is chosen small enough, we can assume that $B(T,\epsilon)_{|F}$ is contained in the union $ U:=\bigcup_{j=1}^m U_j$. 
	Since $(z,u)\in B(T,\epsilon)^{\complement}$ implies that $\dhyp(z,z_{i,j}(u))>\epsilon$, we find the  bound
	\begin{align*}
		||\phi_{|B(T,\epsilon)^{\complement}}||_{\infty} \leq M -N\log(\epsilon)= M +N|\log(\epsilon)|
	\end{align*}
	where 
	\[N:=\sum_{j=1,\dots,m} n_j,\quad M:=\max\{\max_{(z,u)\in \mathring{U}^{\complement}} \phi(z,u), M'' \}-N\log(L) .\]
	Hence the first statement of the proposition is proven.

	In order to prove the second statement of the lemma   we consider a bad point $(w,u)\in T_{|F}$. Since $\epsilon$ is small, we can assume that the tubular neighborhood $B(T,\epsilon)_{|F}$ of $T_{|F}$ is contained in $U=\bigcup_{j=1}^m U_j$. Moreover, if $\epsilon$ is small enough we can also assume that $U_j$ is a product $U_j=K_j\times V_j$ for $K_j\subset F$ and $V_j\subset \PP(\mathcal{V}^{\vee}_c)$ compact subsets. From  now on, for the sake of a simpler notation we set $K:=K_j$ and $V:=V_j$. Moreover we set $P:=P_j$, so that the bad locus $T_{|F}\cap U_j=Z(P)$ is the zero set of $P$.  We will prove the second statement of the lemma  restricting to $U_j=K\times V$. This is sufficient since there are only finitely many $U_j$. 
	
	We need to prove that for any $(w,u)\in Z(P)\subset K\times V$ the following bound holds:
	\[\phi_u(z)_{|B_{\hyp}(w,\epsilon)}\leq M'+N\sum_{z'\in Z(P_u)\cap B_{\hyp}(w,\epsilon)}|\log\left(\dhyp(z,z')\right)| \]
	for some constant $M'$. For any $(z,u)\in K\times V$ we decompose the polynomial $P(z,u)$ into roots 
	\[P(z,u)=\prod_{i=1}^n(z-z_i(u)).\]
	Without loss of generality we can assume that the roots $z_i:V\to \CC$ are well-defined holomorphic functions. Indeed, even if in general they are only multi-valued functions, there is a finite covering $\pi:V'\to V$ such that the pullbacks $\pi^*(z_i):V'\to \CC$ are single valued functions. If we then can prove the desired bound for these pull-back roots, namely if for any root $(w,v)\in Z(\pi^*(P))\subset K\times V'$ it holds
	\[\pi^*(\phi)(z,v)_{|B_{\hyp}(w,\epsilon)}\leq M'+N\sum_{z'\in Z(\pi^*(P)_v)\cap B_{\hyp}(w,\epsilon)}|\log\left(\dhyp(z,z')\right)| \]
	then it is clear that we get the same bound for the original function. Indeed it is enough to choose $v\in \pi^{-1}(u)$ for any $(w,u)\in Z(P)\subset K\times V$ and use the definition of pull-back $\pi^*(\phi)(z,v)=\phi(z,\pi(v))$ and $\pi^*(P)_v=P_{\pi(v)}$ to get the original bound.

	Since we reduced to the case where $P(z,u)=\prod_{i=1}^n(z-z_i(u))$ and   $z_i:V\to \CC$ are holomorphic functions, we can consider the  irreducible components of the zero locus $Z(P)$ which are now given as graphs
	\[\Gamma_i:=\{(z_i(u),u)\in K\times V\colon u\in V\}.\]
	Let us define the tubular neighborhood around $\Gamma_i$ as
	\[B(\Gamma_i):=\{(z,u)\in K\times V\colon \dhyp(z,z_i(u))\leq \epsilon\}.\]
	Notice that by inequality \eqref{eq:partialbound} we have
	\[\phi(z,u)_{|K\times V}\leq M''-\sum_{i=1}^{n}\log|L\cdot \dhyp(z,z_{i}(u))|.\]
	In order to define a global constant $M'$ independent of $u\in V$ in the bound that we are trying to prove, we need to define the maximal compact subset of $B(\Gamma_i)$ where the function $\log|L\cdot \dhyp(z,z_k(u))|_{|B_{\hyp}(z_i(u),\epsilon)}$ is well define and then take the maximum over this set for all $k\not =i$. 
	This set can be defined as  the $k$-th complement
	\begin{equation*}
		\begin{split}
			B(\Gamma_i)(k):&=\{(z,u)\in B(\Gamma_i)\colon  \dhyp(z_i(u),z_k(u))\geq\epsilon\}\\
			&=\{(z,u)\in B(\Gamma_i)\colon  (z_k(u),u)\not \in B(\Gamma_i) \}.
		\end{split}
	\end{equation*}
	The number
	\[M'_{i}:=\max_{k\not =i}\left(\max_{(z,u)\in B(\Gamma_i)(k)} -\log\left(L\cdot\dhyp(z,z_{k}(u))\right)\right)\] 
	is well-defined since the sets we are taking the maximum on are compact and the functions $\log\left(L\cdot\dhyp(z,z_{k}(u))\right)$ are well-defined in these sets.
	We finally define 
	\[M':=N\cdot\max_{i} M'_{i}+M''-N\log(L).\]
	Then by rewriting once again inequality \eqref{eq:partialbound} and using the definition of $M'$ we find that for every  $i\in\{1,\dots,n\}$ it holds 
	\begin{align*}
		\phi(z,u)_{|B(\gamma_i)}&\leq   M'-\sum_{z_{k}(u)\in B_{\hyp}(z_{i}(u),\epsilon)}\log\left(\dhyp(z,z_{k}(u))\right)\\
		&\leq M'+N\sum_{z_{k}(u)\in B_{\hyp}(z_{i}(u),\epsilon)}\left|\log\left(\dhyp(z,z_{k}(u))\right)\right|
	\end{align*} 
	where we added the constant $N>0$ in the second inequality because in the second claim of the proposition we are summing over the bad points not taking multiplicities into account.
	The second claim of the proposition is then proven.
	
\end{proof}

\begin{rem}
	\label{lem:finitebadpoints}
	The positive constant $N>0$ of the previous lemma gives a uniform bound of the number of bad points in any compact fundamental domain for a compact curve  $C$. In the case of a non compact curve, it is unclear if such a uniform bound exists. Notice that  there is an an alternative way  of proving that the number of bad points is uniformly bounded, without using the Weierstrass preparation theorem. Let $\chi_F(z)$ be the characteristic function of a compact fundamental domain $F\subset \HH$. Using a partition of unity argument we can create a continuous function $\tilde{\chi}_F(z)$ that agrees with $\chi_F(z)$ on $F$. Consider then   the continuous map
	\[\PP(\mathcal{V}^{\vee})\to \RR,\quad u\mapsto \int_{\HH}\tilde{\chi}_F(z)\chi_{T(u)}(z)  \]
	where   $\chi_{T(u)}(z)$ is the characteristic function of the set $T(u)$. Since this map is continuous,   it is bounded. The bound is independent of the choice of fundamental domain by the equivariance property
	\[\chi_{T(u)}(\gamma(z))=\chi_{T(\rho(\gamma^{-1})u)}(z),\quad \text{for any } \gamma\in \pi_1(C,c).\]
\end{rem}

Using the bounds of the last lemma, we now prove Proposition \ref{prop:conjcompact} which, as we already noticed,  implies Theorem \ref{thm:conjcompact}. The main strategy is to separate the study of the integral of the seminorm near and far from the bad points.

\begin{proof}[Proof of Proposition \ref{prop:conjcompact}]
	Let us choose $u\in \mathcal{V}^{\vee}$ that computes the Lyapunov exponent, meaning that 
	\[\lambda_1(\mathcal{V})=\lim_{t\to \infty} \frac{1}{t}\frac{1}{2\pi}\int_{0}^{2\pi}\log||u_{g_tr_{\theta}(x_0)}||_{h} \dd \theta\]
	where $x_0\in T^1\HH$ is the base point of $u$.  We used and will use from now on a slight abuse of notation by identifying $g_tr_{\theta}(x_0)\in T^1\HH$ with its base point in $\HH$. This is not a problem since the norms involved are pullbaks of norms for the bundle over $\HH$.
	
	Let us fix
	$\epsilon'>0 \text{ and }  t>>0$ and set 
	\[\epsilon=\epsilon(t):=e^{-t\epsilon'}.\]
	We define now
	\[S(t)_{\text{near}}:=\{\theta\in [0,2\pi]\colon g_tr_{\theta}(x_0)\in   B(T(u),\epsilon)\}\]
	and $S(t)_{\text{far}}:=[0,2\pi]-S(t)_{\text{near}}$.

	We want to prove that  the difference between the norm and the seminorm
	\begin{align*}
		\frac{1}{t}\int_{0}^{2\pi}\log\left(\frac{||u_{g_tr_{\theta}(x_0)}||_{h}}{ ||u_{g_tr_{\theta}(x_0)}||_{\mathcal{L}}}\right)\dd \theta&=\frac{1}{t}\int_{0}^{2\pi}\phi(g_tr_{\theta}(x_0),u)\dd \theta\\
		&=\frac{1}{t}\left(\int_{S(t)_{\text{near}}}\phi(g_tr_{\theta}(x_0),u)\dd \theta+\int_{S(t)_{\text{far}}}\phi(g_tr_{\theta}(x_0),u)\dd \theta\right)
	\end{align*}
	tends to zero.

	We treat first the integral near the bad locus.
	Let $z_0\in \HH$ be the base point of $x_0$ and define the hyperbolic annulus
	\[A(t,\epsilon):=\{z\in \HH\colon t-\epsilon\leq\dhyp(z,z_0)\leq t+\epsilon\}\] and for any $w\in \HH$ define the 
	arc portion
	\[C_t(w,\epsilon):=\{\theta\in [0,2\pi]\colon g_tr_{\theta}(x_0)\in B_{\hyp}(w,\epsilon)\}.\]
	
	It follows from the above definitions  and from the second statement of  Lemma \ref{lem:tbad} that
	\begin{align*}
		&\int_{S(t)_{\text{near}}}\phi(g_tr_{\theta}(x_0),u)\dd \theta\leq \sum_{w\in T(u)\cap A(t,\epsilon)} \int_{C_t(w,\epsilon)}\phi(g_tr_{\theta}(x_0),u)\dd \theta\\
		&\leq \sum_{w\in T(u)\cap A(t,\epsilon)} M'\int_{C_t(w,\epsilon)} \dd \theta+\\
		&\sum_{w\in T(u)\cap A(t,\epsilon)} N\sum_{z'\in T(u)\cap B_{\hyp}(w,\epsilon)}\int_{C_t(w,\epsilon)} \left|\log\left(\dhyp(g_tr_{\theta}(x_0),z')\right)\right| \dd \theta\\
		&\leq \sum_{w\in T(u)\cap A(t,\epsilon)} M'\int_{C_t(w,\epsilon)} \dd \theta+\\
		&\sum_{w\in T(u)\cap A(t,\epsilon)} N^2\max_{z'\in B_{\hyp}(w,\epsilon)}\left(\int_{C_t(w,\epsilon)} \left|\log(\dhyp(g_tr_{\theta}(x_0),z'))\right| \dd \theta\right).
	\end{align*}
	where the last inequality follows since we can assume that $B_{\hyp}(w,\epsilon)$ is small enough to be contained in a fundamental domain and $N$ is the the uniform bound for number of bad points in a fundamental domain. 
	
	\begin{center}
		\begin{tikzpicture}
		\def\r{6pt} 
		\def\factor{3} 
		
		\def\drawangle(#1)(#2)(#3)#4{%
			\pgfmathanglebetweenlines%
			{\pgfpointanchor{#2}{center}}{\pgfpointanchor{#1}{center}}%
			{\pgfpointanchor{#2}{center}}{\pgfpointanchor{#3}{center}}
			\pgfmathround{\pgfmathresult}
			\pgfmathtruncatemacro{\angle}{\pgfmathresult}
			\typeout{Angle #4 = \angle}
			\pgfmathanglebetweenpoints%
			{\pgfpointanchor{#2}{center}}{\pgfpointanchor{#3}{center}}
			\pgfmathtruncatemacro{\endarc}{\pgfmathresult}
			\pgfmathanglebetweenpoints%
			{\pgfpointanchor{#2}{center}}{\pgfpointanchor{#1}{center}}%
			\pgfmathtruncatemacro{\startarc}
			{\pgfmathresult > \endarc ? \pgfmathresult - 360 : \pgfmathresult}
			\draw (#2) ++ (\startarc:\r) arc (\startarc:\endarc:\r);
			\typeout{Start angle: \startarc\space End angle: \endarc}
			\pgfmathsetmacro{\anglenode}{(\endarc+\startarc)/2}
			\pgfmathparse{\factor*\r}
			\path (#2) ++(\anglenode:\pgfmathresult pt) node {$#4$};
		}
		
		\draw (1,1) circle (2.5cm);
		\node [left] at (1,1) {$z_0$};
		\filldraw (1,1) circle[radius=1pt];
		\draw  (1,1)--(1,3.5);
		\node [left] at (1,2.3) {$t$};
		\node [above] at (-1.8,3) {$A(t,\epsilon)$};
		\draw[dashed] (1,1) circle (1.9cm);
		\draw[dashed] (1,1) circle (3.1cm);
		\filldraw (3.2,1) circle[radius=1pt];
		\node [above] at (3.2,1) {$w$};
		\draw (3.2,1) circle (0.6cm);
		\draw  (1,1)--(3.43,1.55);
		\draw  (1,1)--(3.44,0.45);
		
		\coordinate (O) at (1,1); 
		\coordinate (A) at (3.43,1.55); 
		\coordinate (B) at (3.44,0.45); 
		\drawangle(B)(O)(A){}
		\node [right] at (1.3,0.97) {\scriptsize   $C_t(w,\epsilon)$};
		
		\node [left] at (3.45,0.7) {\small $z'$};
		\filldraw (3.3,0.6) circle[radius=0.5pt];
		\end{tikzpicture}
	\end{center}

	Define 
	\[ \theta(\epsilon,t):=\frac{\sinh(\epsilon)}{\sinh(t)}\]
	and notice that
	\[\int_{C_t(w,\epsilon)} \dd \theta\leq\theta(\epsilon,t)\]
	since  the measure of the angle $C_t(w,\epsilon)$ is the same as the quotient of the hyperbolic length of the arc that the angle defines and the hyperbolic length of the circumference $ S_t(z_0)$. It is moreover clear that the hyperbolic length of the arc defined by  $C_t(w,\epsilon)$ is less than the length of the circumference $\partial B_{\hyp}(w,\epsilon)$.
	
	We hence get the following bound:
	\begin{align*}
		&\int_{S(t)_{\text{near}}}\phi(g_tr_{\theta}(x_0),u)\dd \theta \leq M'\cdot \sharp\{ T(u)\cap A(t,\epsilon)\} \cdot\theta(\epsilon,t)+ \\
		&+N^2\cdot \sharp\{ T(u)\cap A(t,\epsilon)\}\cdot \max_{w\in  A(t,\epsilon)}\left( \max_{z'\in B_{\hyp}(w,\epsilon)}\int_{C_t(w,\epsilon)} \left|\log(\dhyp(g_tr_{\theta}(x_0),z'))\right| \dd \theta\right).
	\end{align*}
	
	In order to bound the second summand, we notice that this term is invariant under isometries. Hence we are free to choose a convenient coordinate system and  work for example in the Poincaré disk $\mathcal{D}$ with  starting point  $x_0=(0,(1,0))\in T^1 \mathcal{D}$, namely the center together with the horizontal direction. Moreover, since the term
	\[\max_{z'\in B_{\hyp}(w,\epsilon)}\left(\int_{C_t(w,\epsilon)} \left|\log(\dhyp(g_tr_{\theta}(x_0),z'))\right| \dd \theta\right)\]
	is invariant under rotation, we can assume that $w$ is on the horizontal ray 
	\[w\in R:=A(t,\epsilon)\cap [0,1]=[\tanh(t-\epsilon),\tanh(t+\epsilon)].\]
	We define the tubular neighborhood of the ray as
	\[U:=\bigcup_{r\in [t-\epsilon,t+\epsilon]}B_{\hyp}(\tanh(r/2),\epsilon).\]
	We then get  the following bound:
	\begin{align*}
		&\max_{w\in R}\left(\max_{z'\in  B_{\hyp}(w,\epsilon)}\left(\int_{C_t(w,\epsilon)} \left|\log(\dhyp(g_tr_{\theta}(x_0),z'))\right| \dd \theta\right)\right) \\
		&\leq \max_{z'\in U}\left(\int_{-\frac{\theta(\epsilon,t)}{2}}^{\frac{\theta(\epsilon,t)}{2}} \left|\log\dhyp(\tanh(t/2)e^{i\theta},z')\right| \dd \theta\right)\\
		&\leq \max_{z'\in R}\left(\int_{-\theta(\epsilon,t)}^{\theta(\epsilon,t)} \left|\log\dhyp(\tanh(t/2)e^{i\theta},z')\right| \dd \theta\right)
	\end{align*}
	where the last inequality follows again by rotational invariance (up to enlarging the angle of integration we can assume that $z'$ is  on the ray that cuts the angle in two equal parts and then we can rotate to have $z'\in R$).
	Notice now that 
	
	\begin{align*}
		& \max_{z'\in R}\left(\int_{-\theta(\epsilon,t)}^{\theta(\epsilon,t)} \left|\log\dhyp(\tanh(t/2)e^{i\theta},z')\right| \dd \theta\right)\\
		&\leq \int_{-\theta(\epsilon,t)}^{\theta(\epsilon,t)} \left|\log\dhyp(\tanh(t/2)e^{i\theta},R)\right| \dd \theta
	\end{align*}
	where by definition
	\[\dhyp(\tanh(t/2)e^{i\theta},R):=\inf_{z'\in R}\dhyp(\tanh(t/2)e^{i\theta},z').\]
	
	Using  the hyperbolic sine rule we get
	\[\dhyp(\tanh(t/2)e^{i\theta},R)=\sinh^{-1}(\sinh(\tanh(t/2))\cdot \sin(\theta))=\theta +o(\theta)\]
	where the little-$o$ notation is with respect to $t$ going to infinity (hence $\theta=\theta(t)$ going to zero).
	We finally can then rewrite 
	\begin{align*}
		& \int_{-\theta(\epsilon,t)}^{\theta(\epsilon,t)} \left|\log\dhyp(\tanh(t/2)e^{i\theta},R)\right| \dd \theta= 2\int_{0}^{\theta(\epsilon,t)} \left|\log\theta\right|\dd \theta + o(\theta(\epsilon,t))\\
		&=\left|2\theta(\epsilon,t)\log\left(\theta(\epsilon,t)\right)-2\theta(\epsilon,t)\right|+o(\theta(\epsilon,t))\sim \log(\theta(\epsilon,t))\theta(\epsilon,t)\sim\frac{\epsilon(t)\log(\epsilon(t)/e^t)}{e^{t}}.
	\end{align*}
	The last asymptotic follows from the original definition of $\epsilon=\epsilon(t)=e^{-t\epsilon'}$, for a fixed  small $\epsilon'>0$, and the definition of $\theta(\epsilon,t)$, which together with the fact that for large $t$  the function $\sinh(t)$ is asymptotic to $e^t$ and for $\epsilon$ small $\sinh(\epsilon)$ is asymptotic $\epsilon$, gives 
	\[\theta(\epsilon,t)=\frac{\sinh(\epsilon)}{\sinh(t)}\sim \frac{\epsilon(t)}{e^t} .\]

	Putting together all the inequalities we showed that 
	\begin{align*}
		&\lim_{t\to \infty} \frac{1}{t} \int_{S(t)_{\text{near}}}\phi(g_tr_{\theta}(x_0),u)\dd \theta \leq \liminf_{t\to \infty} \frac{1}{t} \sharp\{ T(u)\cap A(t,\epsilon)\} \cdot \frac{\epsilon(t)\log(\epsilon(t)/e^t)}{e^{t}}\\
		&\leq \liminf_{t\to \infty}  \sharp\{ T(u)\cap D_{t+\epsilon}(z_0)\} \cdot \frac{\epsilon(t)\log(\epsilon(t)/e^t)}{te^{t}}.
	\end{align*}
	By Theorem \ref{thm:lyapexp}, we know that the limit defining the error term converge:
	\[\lim_{T\to \infty}  \frac{1}{T}\int_{0}^{T}\frac{\sharp \{T(u)\cap D_{t+\epsilon}(z_0)\}}{\vol(D_{t+\epsilon}(z_0))} \dd t=\lim_{T\to \infty}  \frac{1}{T}\int_{0}^{T}\frac{\sharp \{T(u)\cap D_{t+\epsilon}(z_0)\}}{4\pi \sinh^2((t+\epsilon)/2)} \dd t<\infty.\]
	This implies that
	\[ \liminf_{t\to \infty}  \sharp\{ T(u)\cap D_{t+\epsilon}(z_0)\} \cdot \frac{\epsilon(t)\log(\epsilon(t)/e^t)}{te^{t}}=0.\]
	Indeed if this is not the case then there is a constant $c>0$ and $t'$ such that for all $t>t'$ it holds $\displaystyle \sharp\{ T(u)\cap D_{t+\epsilon}(z_0)\}>c\cdot(te^t/(\epsilon(t)\log(\epsilon(t)/e^t) )$ which implies that
	\[\lim_{T\to \infty}  \frac{1}{T}\int_{t'}^{T}\frac{\sharp \{T(u)\cap D_{t+\epsilon}(z_0)\}}{4\pi \sinh^2((t+\epsilon)/2)} \dd t>\lim_{T\to \infty}  \frac{1}{T}\int_{t'}^{T}\frac{c\cdot(te^t/(\epsilon(t)\log(\epsilon(t)/e^t) )}{4\pi \sinh^2((t+\epsilon)/2)} \dd t.\]
	This yields to a contradiction, since the right-hand side of the previous inequality is not finite. Indeed the integrand is asymptotic to
	\[	\displaystyle \frac{te^t}{\epsilon(t)\log(\epsilon(t)/e^t)e^{t+\epsilon(t)}}\sim \frac{te^{t\epsilon'}}{\log(e^{-t(1+\epsilon')})}\overset{t\to \infty}{\longrightarrow} \infty.\]
	
	Using the first statement of  Lemma \ref{lem:tbad} we can compute a bound of the integral over the points which are not near $T(u)$:
	\[\lim_{t\to \infty}\frac{1}{t}\int_{S(t)_{\text{far}}}\phi(g_tr_{\theta}(x_0),u)\dd \theta\leq \lim_{t\to \infty}\frac{2\pi}{t}(M-N(\log(\epsilon)) = 2\pi N\epsilon'. \]
	
	By letting $\epsilon'$ tend to zero, we finally get 
	\[\lambda_1(\mathcal{V})=\lim_{T\to \infty} \frac{1}{T}\frac{1}{2\pi}\int_{0}^{2\pi} \log\|G_T r_{\theta} u\|_{\mathcal{L}}\dd\theta\] 
	and Proposition \ref{prop:conjcompact} is proven.
	
\end{proof}

\subsection{Harmonic measures and Brownian motion}

In this section we compare Theorem \ref{thm:conjcompact} to the main results of \cite{deroin} and \cite{deroindaniel}. We will describe how in the case of compact base curve Theorem \ref{thm:conjcompact}  can be used to identify the dynamical degree defined in \cite{deroindaniel} with our error term, which in turn can be viewed as a generalization of the asymptotic covering degree of developing maps defined in \cite{deroin}.

In   \cite{dd1} and \cite{deroin}, Deroin and Dujardin defined the Lyapunov exponents associated to holonomies of parabolic projective structure on hyperbolic surfaces in the context of Brownian motion. The definition of Lyapunov exponents in this context is essentially the same as our definition provided by Oseledets  multiplicative ergodic theorem, but the cocycle is defined over the Brownian motion on the Riemann surface instead of on the geodesic flow. 
The two definitions of Lyapunov exponents provide the same numbers  since the Brownian motion tracks the geodesic flow sublinearly on hyperbolic Riemann surfaces (see \cite{ancona}).
In \cite{dd1} it was proven the main equality of Theorem \ref{thm:conjcompact} in the specific case of rank 2 representations given as holonomies of projective structures inducing the same holomorphic structure of the base curve.  In Proposition \ref{prop:operproj} we identified the locus of such representations with the oper locus if the base curve is compact. The error term was identified with the asymptotic covering degree of the developing map of the projective structure (the different constants appearing are due to a different normalization of the hyperbolic metric). Our error term is a generalization of this asymptotic degree (see Proposition \ref{prop:rank2shatz} for a specific comparison in rank two).  Notice however  that the equality proven in \cite{dd1} is more general than ours since it works for parabolic representations over non compact curves and since in the error term they do not need the integral defining the mean of the counting function since they can prove that the counting function converges.

In \cite{deroindaniel}, Daniel and Deroin generalized the definition of Lyapunov exponents in the context of Brownian motion on Kähler manifolds. The result they provide is analogous to the main equality of Theorem \ref{thm:conjcompact}, where the error term is called dynamical degree.  Recall that  a measure $\nu$ on the projective bundle associated to a flat bundle $\mathcal{V}$ is called  harmonic  if it is invariant under the heat semigroup action. The dynamical degree associated to a sub-vector bundle $\mathcal{E}\subset \mathcal{V}$ is defined as the intersection number
\[\delta_{\mathcal{E}}:=T_{\nu} \cap [\PP(\mathcal{E})] \]
where $T_{\nu}$ is the harmonic current associated to $\nu$.  

\begin{cor}
	\label{cor:dynamdegree}
	Over a compact Riemann surface, the error term and the dynamical degree coincide
	\[\delta_{\mathcal{E}}= \err^{\mathcal{E}}(u)\]
	for Lebesgue almost all $u\in \bigwedge^{k} \mathcal{V}^{\vee}$.
\end{cor}
\begin{proof}
	As recalled above, the Lyapunov exponents defined in the context of Brownian motion and the one defined for the geodesic flow coincide on a hyperbolic curve.
	The result then follows by comparing the equality of Theorem \ref{thm:conjcompact} and the formula in \cite[Theorem 5]{deroindaniel}.
\end{proof}

Notice that the error term, contrary to the dynamical degree, can in principle be approximated with computer experiments.

Recall that $G_t:\mathcal{V}\to \mathcal{V}$ is the lift via parallel transport of the  geodesic flow over $T^1C$.
The main theorem of \cite{eskinwilkinson} in our setting implies  that if the flat bundle $\mathcal{V}_C$ over $T^1C$ is irreducible, then there exists a measure $\nu$ on the projective bundle $\PP(\mathcal{V}_C)$  that is $G_t$-invariant, projects to the hyperbolic measure on the base and it is fiberwise supported on the projectivization $\PP(\mathcal{V}_{\lambda_1})$ of the first Oseledets  subspace.
Using an abuse of notation we call $\nu$ the corresponding measure on the wedge products $\PP\left(\bigwedge^k\mathcal{V}_C\right)$ if they are irreducible.  \\
The property of our  error term to be  Lebesgue almost everywhere constant differentiate our result to the one in \cite{deroindaniel}. If we allow ourselves to consider an error term which is almost everywhere constant with respect to a $G_t$-invariant measure $\nu$, we are able to more easily show a weaker version of Theorem \ref{thm:conjcompact}  equivalent to the result of \cite{deroindaniel}.
\begin{prop}
	\label{prop:eq}
	Let $\mathcal{V}_C$ be a $k$-irreducible flat bundle.
	For any  holomorphic subbundle   $\mathcal{E}\subset \mathcal{V}_C $ of rank $k$, if \[\int_{\PP(\mathcal{V}^{\vee})} \log\left(\frac{||u||_{\mathcal{L}} }{ ||u||_h}\right)\dd \nu(u)<\infty\] then 
	\begin{equation*}
		\sum_{i=0}^k \lambda_i= \frac{2\pardeg(\Xi_h(\mathcal{E}))}{\deg(\Omega_{\overline{C}}^1(\log(\Delta))}  + \err^{\mathcal{E}}(u)
	\end{equation*}
	for $\nu$-almost any $u\in \PP\left(\bigwedge^k \mathcal{V}^{\vee}\right)$.\\
	If the base curve is compact, the integrability assumption  holds.
\end{prop}


\begin{proof}[Proof of Proposition \ref{prop:eq}]
	As in the proof of Theorem \ref{thm:lyapexp}, it is enough to prove the result for the top Lyapunov exponent and consider the case where $\mathcal{E}=\mathcal{L}$ is a line bundle.
	Since the measure $\nu$ is fiberwise supported on the first Oseledets  space, the top Lyapunov exponent of $\mathcal{V}$ is given ny 
	\[\lambda_1=\lim_{t\to \infty} \frac{1}{t} \int_{\PP(\mathcal{V}^{\vee})} \log\left(\frac{||G_t u||_h}{||u||_h}\right)\dd \nu(u)\]
	where $h$ is an integrable norm.
	Integrating inequality \eqref{eq:bounderr} over $\PP(\mathcal{V}^{\vee})$ with respect to the measure $\nu$ and rewriting backwards the equalities of the proof of Theorem \ref{thm:lyapexp}, we find 
	\[\lambda_1\geq \frac{2\pardeg(\Xi_h(\mathcal{E}))}{\deg(\Omega_{\overline{C}}^1(\log(\Delta))}  + \int_{\PP(\mathcal{V}^{\vee})}\err^{\mathcal{E}}(u) \dd\nu(u)=\lim_{t\to \infty}\frac{1}{t}\int_{\PP(\mathcal{V}^{\vee})}\log\left(\frac{||G_t u||_{\mathcal{L}}}{||u||_{\mathcal{L}}}\right)\dd \nu(u)\]
	where $||\cdot||_{\mathcal{L}}$ is the $\mathcal{L}$-seminorm defined in \eqref{eq: seminorm}. 
	We finally compute
	\begin{align*}
		&\lambda_1-\left(\frac{2\pardeg(\Xi_h(\mathcal{E}))}{\deg(\Omega_{\overline{C}}^1(\log(\Delta))}  + \int_{\PP(\mathcal{V}^{\vee})}\err^{\mathcal{E}}(u) \dd\nu(u)\right)\\
		&=\lim_{t\to \infty} \frac{1}{t} \int_{\PP(\mathcal{V}^{\vee})} \log\left(\frac{||G_t u||_h}{||u||_h}\right)\dd \nu(u)-\lim_{t\to \infty}\frac{1}{t}\int_{\PP(\mathcal{V}^{\vee})}\log\left(\frac{||G_t u||_{\mathcal{L}}}{||u||_{\mathcal{L}}}\right)\dd \nu(u)\\
		&=\lim_{t\to \infty} \frac{1}{t} \int_{\PP(\mathcal{V}^{\vee})} \log\left(\frac{||G_t u||_h \cdot||u||_{\mathcal{L}} }{||G_t u||_{\mathcal{L}}\cdot ||u||_h}\right)\dd \nu(u).
	\end{align*}
	Since by assumption 
	$\int_{\PP(\mathcal{V}^{\vee})} \log\left(\frac{||u||_{\mathcal{L}} }{ ||u||_h}\right)\dd \nu(u)<\infty$,
	we can split the integral
	\begin{align*}
		&\int_{\PP(\mathcal{V}^{\vee})} \log\left(\frac{||G_t u||_h \cdot||u||_{\mathcal{L}} }{||G_t u||_{\mathcal{L}}\cdot ||u||_h}\right)\dd \nu(u)=\\
		&=\int_{\PP(\mathcal{V}^{\vee})} \log\left(\frac{||G_t u||_h }{||G_t u||_{\mathcal{L}}}\right)\dd \nu(u)-\int_{\PP(\mathcal{V}^{\vee})} \log\left(\frac{||u||_h }{ ||u||_{\mathcal{L}}}\right)\dd \nu(u)=0
	\end{align*}
	
	where the last equality follows from the $G_t$-invariance of the measure $\nu$. We have then proved that
	\begin{align*}
		\lambda_1&=\left(\frac{2\pardeg(\Xi_h(\mathcal{E}))}{\deg(\Omega_{\overline{C}}^1(\log(\Delta))}  + \int_{\PP(\mathcal{V}^{\vee})}\err^{\mathcal{E}}(u) \dd\nu(u)\right)=\\
		&=\lim_{t\to \infty}\frac{1}{t}\int_{\PP(\mathcal{V}^{\vee})}\log\left(\frac{||G_t u||_{\mathcal{L}}}{||u||_{\mathcal{L}}}\right)\dd \nu(u).
	\end{align*}
	
	In order to prove that the function $\err^{\mathcal{E}}(u)$ is $\nu$-almost everywhere constant notice that the function
	\[\PP(\mathcal{V}^{\vee})\longrightarrow \RR,\quad u\mapsto \log\left(\frac{||G_1 u||_{\mathcal{L}}}{||u||_{\mathcal{L}}}\right)\]
	is $\nu$-integrable since 
	\[\lambda_1=\lim_{t\to \infty}\frac{1}{t}\int_{\PP(\mathcal{V}^{\vee})}\log\left(\frac{||G_t u||_{\mathcal{L}}}{||u||_{\mathcal{L}}}\right)\dd \nu(u)=\int_{\PP(\mathcal{V}^{\vee})}\log\left(\frac{||G_1 u||_{\mathcal{L}}}{||u||_{\mathcal{L}}}\right)\dd \nu(u)\]
	where  the second equality comes from the $G_t$-invariance of $\nu$.
	Then applying the Birkhoff ergodic theorem to this function and the measure $\nu$, it follows that for $\nu$-almost any $u\in \PP(\mathcal{V}^{\vee})$:
	\[\lim_{t\to \infty} \frac{1}{t}\log\left(\frac{||G_t u||_{\mathcal{L}}}{||u||_{\mathcal{L}}}\right)=\int_{\PP(\mathcal{V}^{\vee})}\log\left(\frac{||G_1 u||_{\mathcal{L}}}{||u||_{\mathcal{L}}}\right)\dd \nu(u)=\lambda_1.\] 
	Hence we finally get for $\nu$-almost any $u\in \PP(\mathcal{V}^{\vee})$:
	\begin{align*}
		\lambda_1=\lim_{t\to \infty} \frac{1}{t}\log\left(\frac{||G_t u||_{\mathcal{L}}}{||u||_{\mathcal{L}}}\right)=\frac{2\pardeg(\Xi_h(\mathcal{E}))}{\deg(\Omega_{\overline{C}}^1(\log(\Delta))}  + \err^{\mathcal{E}}(u).
	\end{align*}
	
	The second claim of the proposition about the integrability condition in the compact case can be showed using Lemma \ref{lem:tbad} and arguing in an analogous way as in the proof of Proposition \ref{prop:conjcompact}.
\end{proof}


\section{Lyapunov exponents for variations of Hodge structures}
\label{sec:vhsrat}

In this section we use the previous results in order to describe the properties of Lyapunov exponents for special flat bundles, namely variations of Hodge structures.

First of all we  we use the geometric version of Oseledets theorem described in \cite{simionzero} to give a bound on the number of zero exponents of a general  variation of Hodge structures. We then use Simpson correspondence to classify unitary representations as variations of Hodge structures  with trivial Lyapunov spectrum and deduce from the proof of this last statement a simplicity result for non-unitary variations of Hodge structures.  

We finally use  the condition given by Proposition \ref{prop:conditionemptylocus} to prove  results about rationality of Lyapunov exponents for variations of Hodge structures in low weight and describe what is known in these situations. More specifically, we will prove a slightly generalized version of the results of \cite{ekz} and \cite{simionk3} about equality of Lyapunov exponents and degrees of Hodge bundles in the weight 1 and in the real weight 2 variation of Hodge structures case.
The core of the arguments of the proofs that we present are similar as the one in the original proofs of the papers cited above or the ones of the most recent paper \cite{deroindaniel}. The main  idea is to relate the Lyapunov exponents to properties of the period maps. In particular, we will use that the image of the period map cannot contain the point corresponding to some Oseledets  spaces in the two cases of weight one and real K3 type variation of Hodge structures.

We recall the definition of variation of Hodge structures. They are special holomorphic flat bundles arising for example from the variation of  the cohomology  of families of algebraic varieties.

\begin{defn}
	\label{def:vhs}
	A complex variation of Hodge structures of weight $k$ over $C$ is a holomorphic flat vector bundle $(\mathcal{V},\nabla)$ over $C$ together with a holomorphic filtration
	\[F^{k+1}=0\subset \dots \subset F^0=\mathcal{V}\]
	which satisfies the Griffiths transversality condition 
	\[\nabla:F^p\to F^{p-1}\otimes \Omega^1_C\]
	and such that furthermore there exists a $\nabla$-flat hermitian complex form $H$ on $\mathcal{V}$,  which is positive definite on $F^i/F^{i+1}$ for $i$ even and negative definite for $i$ odd.
\end{defn}

Via Simpson correspondence, variations of Hodge structures  correspond to stable systems of Hodge bundles (\cite{simvhs}).

\begin{defn}
	A system of Hodge bundles is a Higgs bundle $(\mathcal{V},\Phi)$ together with a decomposition $\mathcal{V}=\oplus \mathcal{V}^{p,q}$, such that $\Phi: \mathcal{V}^{p,q}\to  \mathcal{V}^{p-1,q+1}\otimes \mathcal{K}_C$.
\end{defn}

Via Simpson correspondence a variation of Hodge structures is associated to the semistable system of Hodge bundles given by the graded object associated to the Hodge filtration equipped with the Higgs field defined by the graded pieces of the flat connection. In this case the harmonic metric is the Hodge metric coming from the hermitian form $H$. 

Finally we want to remark special properties of weight one and real weight two variations of Hodge structures what will be interesting to relate to Lyapunov exponents (see  Section \ref{sec:vhsrat}). A real weight two  variation of Hodge structures   is given by a real vector bundle $\mathcal{V}_{\RR}$ over $C$ such that its base change to $\CC$ defines a complex weight 2 variation of Hodge structures $\mathcal{V}$.

\begin{prop}
	\label{prop:w1k3HN}
	\begin{enumerate}
		\item Let $\mathcal{V}$ be a weight one variation of Hodge structures.  Then the first piece of the Hodge  filtration $F^1=\mathcal{V}^{1,0}$ is the maximal degree subbundles among all subbundles of $\mathcal{V}$.
		\item Let  $\mathcal{V}$ be a real variation of Hodge structures of weight two and let $\mathcal{V}^{2,0}$ be the first piece of the Hodge filtration. Then for every subbundle  $\mathcal{E}\subseteq \mathcal{V}$  it holds $\deg(\mathcal{E})\leq 2\deg(\mathcal{V}^{2,0})$.
	\end{enumerate}	
\end{prop}

\begin{proof}
	First of all we remark that in the proof we will work with subsheaves of holomorphic vector bundles, and not only with subbundles. All the semistability arguments still work in this more general context since we can always pass to the saturation of the subsheaves.
	
	If $\mathcal{V}$ is a weight one variation of Hodge structures, the associated semistable system of Hodge bundles is  given by $\mathcal{V}^{1,0}\oplus \mathcal{V}^{0,1}$, where $\mathcal{V}^{0,1}=\mathcal{V}/\mathcal{V}^{1,0}$. First of all note that  if $\mathcal{W}\subseteq \mathcal{V}^{1,0}$, then $\deg(\mathcal{W})\leq \deg(\mathcal{V}^{1,0})$. Indeed $\mathcal{W}\oplus \mathcal{V}^{0,1}$ is sub-system of Hodge bundles. Hence by semistability $\deg(\mathcal{W})+\deg(\mathcal{V}^{0,1})\leq0$.
	If $\mathcal{E}\subset \mathcal{V}$ is any subbundle, then consider the short exact sequence
	\[0\to \mathcal{E}\cap  \mathcal{V}^{1,0}\to\mathcal{E}\to \frac{\mathcal{E}}{\mathcal{E}\cap  \mathcal{V}^{1,0}}\to 0. \]
	Since the quotient $\frac{\mathcal{E}}{\mathcal{E}\cap  \mathcal{V}^{1,0}}$ is a subsheaf of $\mathcal{V}^{0,1}$, it defines a sub-system of Hodge bundles and so it has negative degree. By additivity of the degree we finally get 
	\[\deg(\mathcal{E})=\deg(\mathcal{E}\cap  \mathcal{V}^{1,0})+\deg(\frac{\mathcal{E}}{\mathcal{E}\cap  \mathcal{V}^{1,0}})\leq \deg(\mathcal{V}^{1,0}).   \]
	
	If $\mathcal{V}$ is a real variation of Hodge structures of weight two, the associated semistable system of Hodge bundles is given by $ \mathcal{V}^{2,0}\oplus \mathcal{V}^{1,1}\oplus \mathcal{V}^{0,2}$ where by definition $\mathcal{V}^{1,1}=F^1/\mathcal{V}^{2,0}$ and $ \mathcal{V}^{0,2}= \mathcal{V}/F^1$. Let now $\mathcal{E}\subseteq \mathcal{V}$ be a subbundle. First of all consider the bundle $\frac{\mathcal{E}}{\mathcal{E}\cap F^1}$. Since it injects as a subsheaf of $ \mathcal{V}^{0,2}$, it defines a sub-Higgs sheaf of the associated system of Hodge bundles. By semistability it has then to have negative degree and so we get 
	\[\deg(\mathcal{E})\leq \deg(\mathcal{E}\cap F^1). \]
	Moreover, since  $(\mathcal{E}\cap \mathcal{V}^{2,0})\oplus  \mathcal{V}^{1,1}\oplus \mathcal{V}^{0,2}$ also defines a sub-Higgs sheaf and since $\deg(\mathcal{V}^{0,2})=-\deg(\mathcal{V}^{2,0})$ and $\deg(\mathcal{V}^{1,1})=0$ (this is because $ \mathcal{V}$ is a real VHS), again by semistability  we obtain 
	\[\deg(\mathcal{E}\cap \mathcal{V}^{2,0})\leq \deg(\mathcal{V}^{2,0}).\]
	To conclude consider the sub-Higgs sheaf $\frac{\mathcal{E}\cap F^1}{\mathcal{E}\cap \mathcal{V}^{0,2}}\oplus \mathcal{V}^{0,2}$, which by semistability has negative degree. It follows that 
	\[ \deg(\mathcal{E}\cap F^1)\leq \deg(\mathcal{E}\cap \mathcal{V}^{0,2})+\deg(\mathcal{V}^{2,0}).\]
	Using the three inequalities that we obtained we get
	\[\deg(\mathcal{E})\leq  \deg(\mathcal{E}\cap F^1)\leq\deg(\mathcal{E}\cap \mathcal{V}^{0,2})+\deg(\mathcal{V}^{2,0})\leq 2\deg(\mathcal{V}^{2,0}).  \]
\end{proof}

The last proposition will be related  to the characterization of Lyapunov exponents from the main equality of Theorem \ref{thm:ekmz} and Theorem \ref{thm:simionk3}.

\subsection{Oseledec decomposition for variations of Hodge structures}



 The monodromy invariance  of the indefinite Hermitian form $H$  allows to derive the well-known orthogonality of Oseledets subspaces for variations of Hodge structures. This is an important ingredient that distinguishes variations of Hodge structures to generic flat bundles.

\begin{prop}
	\label{prop:Oseledets istrop}
	The Oseledets  subspace $\mathcal{V}_{\lambda_i}$ are totally isotropic with respect to the $\nabla$-flat indefinite hermitian form $H$ unless $\lambda_i=0$ and they are pairwise orthogonal unless $\lambda_i=-\lambda_j$.
\end{prop}
We recall the proof of this statement, which can be found for example in \cite{martinandre}.
\begin{proof}
	Let $K$ be a positive measure compact subset of $T^1C$, let ${t_k}$ be a sequence going to infinity such that $g_{t_k}(c)\in K$ for $t_k\to \pm \infty$ and for almost all $c\in T^1C$. The existence of the sequence is justified by Poincaré recurrence theorem. Let $u_i\in \mathcal{V}_{\lambda_i}$. Then by $G_t$-invariance of $H$ and by the Cauchy-Schwartz inequality we get
	\[H(u_i,u_j)=H(g_{t_k}u_i,g_{t_k}u_j)\leq c(K)||g_{t_k}u_i||_H||g_{t_k}u_j ||_H\sim e^{(\lambda_i+\lambda_j)\cdot t_k}\]
	where $c(K)>0$ is a positive constant depending only on $K$. Since  $e^{(\lambda_i+\lambda_j)\cdot t_k}\to 0$ for  $\lambda_i\not =\lambda_j$ and for $t_k\to \infty$ or $t_k\to -\infty$, we get the result.
\end{proof}

\subsection{Zero exponents and simplicity results for variation of Hodge structures.}

First of all we use a well-known version of the geometric Oseledets theorem (\cite{simionzero}) to give a bound on the number of zeros of variation of Hodge structures.

\begin{prop}
	\label{prop:zeroexp}
	Let $(\mathcal{V},\nabla)$ be a rank $n$ irreducible complex variation of Hodge structures of   weight $k$ with Hodge filtration $(F^i)$. Let $p=\sum_{i\equiv 0(2)} \rk\left( F^i/F^{i+1} \right)$ be the rank of the even part of $\mathcal{V}$.  Then there are at least $|n-2p|$ zero Lyapunov exponents.
\end{prop}
\begin{proof}
	By the geometric Oseledec theorem,   any vector bundle coming from a $SU(p,q)$-flat bundle has at least $|p-q|$ zero exponents (\cite[Ex.4.5]{simionzero}). We can apply this statement to complex variations of Hodge structures, since by definition the image of the monodromy representation is in $SU(p,q)$ where $p=\sum_{i\equiv 0(2)} h^{i,j}$ is the dimension of the even Hodge bundles and $q=\sum_{i\equiv 0(2)} h^{i,j}$ is the dimension of the odd ones.
\end{proof}

We now prove that for a variation of Hodge structure over a compact curve,  triviality of the Lyapunov spectrum is equivalent to having weight zero, which means corresponding to a unitary representation.

\begin{prop}
	\label{prop:weight0}
	Let $(\mathcal{V},\nabla)$ be a rank $n$ irreducible varation of Hodge structure. Then the following are equivalent:
	\begin{enumerate}
		\item $\mathcal{V}$ is a variation of Hodge structures of  weight  zero.
		\item the corresponding monodromy representation is unitary.
		\item $\mathcal{V}$ is stable.
	\end{enumerate} 
	Moreover the conditions of above are equivalent to have trivial Lyapunov spectrum.
\end{prop}

\begin{proof}
	Recall that by the result of Narasimhan-Seshadri (Theorem \ref{thm:n-s}), the locus of unitary representations corresponds to  the locus of stable vector bundles $\mathcal{V}$ equipped with the harmonic metric connection or, by Simpson correspondence, to the locus of  Higgs bundles with zero Higgs field. 
	If a variation of Hodge structures is of weight zero, then the Griffiths filtration is trivial and so the associated system of Hodge bundles has zero Higgs field. So the associated representation is unitary. Moreover, again by Narasimhan-Seshadri's result, if the associated representation is unitary, then $\mathcal{V}$ is stable.\\
	We will prove now that stability implies weight zero. 
	Assume by contradiction that the weight is bigger than zero, so that we have a non-trivial filtration 
	\[F^{k+1}=0\subset \dots \subset F^0=\mathcal{V}\]
	corresponding to the system of Hodge bundles
	\[(\Gr_F(\mathcal{V})=\bigoplus_p \mathcal{V}^{p,n-p}, \nabla^{gr}),\ \mathcal{V}^{p,n-p}:=F^p/F^{p+1}.\]
	By Griffiths transversality condition, $(\mathcal{V}^{0,n},\nabla^{gr}_{|\mathcal{V}^{0,n}}=0)\subset (\Gr_F(\mathcal{V}),\nabla^{gr}) $ is a sub-Higgs bundle, and so by semistability of Higgs bundles in $\higgs$, we have that $\mu(\mathcal{V}^{0,n})\leq \mu(\Gr_F(\mathcal{V}))=0$.
	Since $ \deg(\Gr_F(\mathcal{V}))=\sum_p \deg(\mathcal{V}^{p,n-p})$, it follows that
	\[\deg(\mathcal{V}^{0,n})=-\deg(\bigoplus_{p\not =0}  \mathcal{V}^{p,n-p})=-\deg(F^1)\leq 0.\]
	So we have found that  $\deg(F^1)\geq 0=\mu(\mathcal{V})$ which is impossible by the stability of $\mathcal{V}$.\\
	We prove now the last statement of the proposition. We know that a unitary representation has trivial Lyapunov spectrum since in this case the norm used to compute the Lyapunov exponent is invariant under parallel transport. Conversely, if the Lyapunov spectrum is zero then, since $F^1\subset \mathcal{V}$ is a holomorphic subbundle, we have that  $0=\sum_{i=1}^{\rk(F_1)}\lambda_i\geq \deg(F^1)$ by Theorem \ref{thm:ekmz}.  But as we already saw, if the weight is positive then it holds also $0\leq \deg(F^1)$. If $F^1$ has zero degree then it defines a flat subbundle, contradicting the hypothesis of irreducibility of $\mathcal{V}$.
\end{proof}

From the last Theorem we get a direct corollary about the sum of the first  $\rk(F_1)$ Lyapunov exponents of a positive weight variation of Hodge structures.

\begin{cor}
	\label{cor:boundvhs}
	Let $(\mathcal{V},\nabla)$ be a rank $n$ irreducible variation of Hodge structures of  positive weight. Then the the sum of the first  $\rk(F^1)$ Lyapunov exponents is positive and the following non-trivial bound holds:
	\[\sum_{i=1}^{\rk(F_1)}\lambda_i(\mathcal{V})\geq \deg(F^1)=\deg(\mathcal{V}^{n,0})>0.\]
\end{cor}
\begin{proof}
	From the proof of the last Theorem we see that if the weight is positive and the variation of Hodge structures is irreducible then the degree of $F^1$ is strictly positive. Hence the  bound \eqref{eq:bound} gives the result.
\end{proof}


\subsection{Weight 1 variation of Hodge structures}

Recall that a weight 1 complex variation of Hodge structure of rank $n$ over $C$  is given by a flat vector bundle $\mathcal{H}$ of rank $n$ together with a holomorphic subbundle $\mathcal{H}^{1,0}\subset \mathcal{H}$ of rank $k$ and a $\nabla$-flat hermitian complex form $H$  on $\mathcal{H}$ that is positive definite on  $\mathcal{H}^{1,0}$ and negative definite on $\mathcal{H}/\mathcal{H}^{1,0}$.

We can reprove the result of  \cite{ekz} in the case of complex weight $1$ variation of Hodge structures over hyperbolic curves.

\begin{thm}
	\label{thm:ekz}
	If $\mathcal{H}$ is a weight $1$ complex variation of Hodge structures of rank $n$ with $\rk(\mathcal{H}^{1,0})=k$ over $C$, then
	\[\sum_{i=1}^k \lambda_i=\frac{2\pardeg(\Xi_h(\mathcal{H}^{1,0}))}{\deg(\Omega_{\overline{C}}^1(\log(\Delta))} \]
\end{thm}

\begin{proof}
	
	We want to use Proposition \ref{prop:conditionemptylocus} and prove that for any vector $u$ of the $G_t$-invariant closed subspace $\mathcal{S}\subset \PP(\bigwedge^k\mathcal{H}^{\vee})$ given by totally isotropic $(n-k)$-planes  the bad locus $\displaystyle \tbad^{\mathcal{H}^{1,0}}(u)$ is empty. This locus is $G_t$-invariant because the indefinite metric $H$ is $G_t$-invariant. Moreover, we need to prove that there is a totally isotropic $(n-k)$-plane computing the top $k$-Lyapunov exponents.\\
	Assume that $k\geq n-k$. In order to prove that  the bad locus $\tbad^{\mathcal{H}^{1,0}}(u)$ is empty, notice that the image of the period map 
	\[s_{\mathcal{H}^{1,0}}:\HH\to \\P\left(\bigwedge^k \mathcal{H}_c\right)\]
	is contained in the space of positive definite $k$-planes since $H$ is positive definite on $\mathcal{H}^{1,0}$. Hence every $k$-plane $s_{\mathcal{H}^{1,0}}(z)$ has to intersect trivially any totally isotropic $(n-k)$-plane, which means by condition \eqref{eq:conditionequality}  that the bad locus is empty for any $u\in \mathcal{S}$.
	We now only need to find a  totally isotropic $(n-k)$-plane  computing the top Lyapunov exponents. Since $k\geq n-k$ and the Hermitian form $H$ has signature $(k,n-k)$, using Proposition \ref{prop:Oseledets istrop} it is easy to see that there is a totally isotropic $(n-k)$-plane contained in the positive Oseledets  space $\mathcal{V}_{\geq 0}$. By definition of the Osedelec space this plane computes the sum of the top $k$ Lyapunov exponents.
	
	If $k\leq  n-k$ we consider the complex conjugate variation of Hodge structure. It has the same Lyapunov exponent and now the complex conjugate of the bundle $\mathcal{H}^{0,1}:=\mathcal{H}/\mathcal{H}^{1,0}$ is a holomorphic subbundle with degree equal to $\deg(\mathcal{H}^{0,1})=-\deg(\mathcal{H}^{1,0})$. Now we can use the first part of the proof
	\[\sum_{i=1}^{k} \lambda_i=-\sum_{i=1}^{n-k} \lambda_i=-\frac{2\pardeg(\Xi_h(\mathcal{H}^{0,1}))}{\deg(\Omega_{\overline{C}}^1(\log(\Delta))}=\frac{2\pardeg(\Xi_h(\mathcal{H}^{1,0}))}{\deg(\Omega_{\overline{C}}^1(\log(\Delta))} .\] 
\end{proof}

Notice that as a corollary of the last theorem we get back the first part of the statement of Proposition \ref{prop:w1k3HN}, namely that $\mathcal{H}^{1,0}$ is the maximal degree subbundle among all the subbundles of $\mathcal{H}$. 


\subsection{Real variations of Hodge structures of K3 type}

A weight two real variation of Hodge structures  over $C$   is given by a real vector bundle $\mathcal{H}_{\RR}$ over $C$ such that its base change to $\CC$ defines a complex variation of Hodge structures $\mathcal{H}$ of weight two. Let  $F^2=\mathcal{H}^{2,0}\subset F^1  \subset \mathcal{H}$  be the Hodge filtration and  note that $\overline{F^2}=\mathcal{H}^{0,2}:=\mathcal{H}/F^1$. Recall that by definition there is a $\nabla$-flat hermitian complex form $H$  on $\mathcal{H}$ which is positive definite on  $\mathcal{H}^{2,0}$ and $\mathcal{H}^{0,2}$, and negative definite on $\mathcal{H}^{1,1}:=F^1/\mathcal{H}^{2,0}$. Let $n:=\rk(\mathcal{H})$ and $k:=\rk(\mathcal{H}^{2,0})$.

Recall first of all that in the case of real variations of Hodge structures of weight 2, by the geometric Oseledets theorem, and more specifically by Proposition \ref{prop:zeroexp}, there are at least $n-4k$ zero Lyapunov exponents. Hence the Lyapunov spectrum takes the form
\[\lambda_1,\dots,\lambda_{2k},0,\cdots,0,-\lambda_{2k},\dots,-\lambda_1.\]

In \cite{simionk3}, it is proven that for a real variation of Hodge structures of weight 2 of K3 type, i.e. for $k=1$, the bound of Theorem \ref{thm:ekmz} is an equality for the top Lyapunov exponent when we use  the holomorphic sub-line bundle $\mathcal{E}=\mathcal{H}^{2,0}$.
We can reprove this result using Proposition \ref{prop:conditionemptylocus}. 

\begin{thm}
	\label{thm:simionk3}
	Let $\mathcal{H}$ be a flat vector bundle of rank $n$ corresponding to a  real variation of Hodge structures of  weight two of K3 type, namely $\rk(\mathcal{H}^{2,0})=1$.
	Then 
	\[\lambda_1=\frac{2\pardeg(\Xi_h(\mathcal{H}^{2,0}))}{\deg(\Omega_{\overline{C}}^1(\log(\Delta))}. \]
\end{thm}
\begin{proof}
	As in the proof of weight one case we want to use Proposition \ref{prop:conditionemptylocus}. Hence we will prove that for any vector $u$ in the $G_t$-invariant subset $\mathcal{S}\subset \PP(\mathcal{H}^{\vee})$ given by  hyperplanes whose orthogonal complement is a totally isotropic real line,  the bad locus $\displaystyle \tbad^{\mathcal{H}^{2,0}}(u)$ is empty. Moreover we will prove that there is one hyperplane in $\mathcal{S}$ computing the top Lyapunov exponent.\\
	First of all we  check that all lines in the image of the period map
	\[s_{\mathcal{H}^{2,0}}:\HH\to \PP( \mathcal{H}_c)\]
	do not intersect non-trivially any hyperplane in $ \mathcal{S}$ .
	  	The image of the period map is contained in the period domain defined by positive lines $v\in \PP( \mathcal{H}_c)$ for which $H(v,\overline{v})=0$. Then if there is a hyperplane $u\in \mathcal{S}$ intersecting non trivially a line $v\in \Im(s_{\mathcal{H}^{2,0}})$, it means that the isotropic real line $u^{\perp}$ is contained in the negative definite orthogonal complement $<v,\overline{v}>^{\perp}$, which is impossible.  
	Finally, by Proposition \ref{prop:Oseledets istrop}, the orthogonal complement of the  hyperplane $u=\sum_{i>1}\mathcal{V}_{\lambda_i}$ is the Oseledets  space $\mathcal{V}_{\lambda_n}$ which is a totally isotropic real line. By definition $u$ computes the top Lyapunov exponent.
\end{proof}

Notice that as a corollary of the last theorem we get a stronger  statement than  the one of Proposition \ref{prop:w1k3HN} in the K3 case.
\begin{cor}
	Let $\mathcal{H}$ be a flat vector bundle corresponding to a  real variation of Hodge structures of  K3 type. Then for any sub-line bundle $\mathcal{L}\subset \mathcal{H}$ it holds
	\[\deg(\mathcal{L})\leq \deg(\mathcal{H}^{2,0}).\]
\end{cor}

We believe for the general  two variation of Hodge structures case the situation is analogous to the one described in \cite{ekmz} for Calabi-Yau threefolds. Namely, if the monodromy is arithmetic a strict inequality between the sum of the first $k$ exponents and the normalized degree of $\mathcal{H}^{2,0}$ should always hold.

\section{Lyapunov exponents on De Rham moduli spaces and Shatz stratification}
\label{subsec:R-H}

The set of flat vector bundles can be made into an algebraic variety called de Rham moduli space. We showed in the last section that Lyapunov exponents satisfy special properties on the subset of this moduli space given by variations of Hodge structures. In this section we want to investigate Lyapunov exponents as invariants on the full  moduli space and in particular understand the behavior of these functions on Shatz strata given by Harder-Narasimhan type. Thanks to inequality \eqref{eq:bound} we can give a bound of the Lyapunov exponents functions on these strata. This result is a generalization of the main result of \cite{deroin}, since we will see in the next sections that the maximal stratum in rank two can be identified with the locus considered in \cite{deroin} given by projective structures underlying the same complex structures. We finally use a recent result of \cite{dujfavre} to prove that the Lyapunov exponents functions are unbounded on the maximal stratum, which means in particular that they are unbounded on the moduli space.

More specifically the De Rham moduli space $\drham$  is the moduli space of semisimple flat vector bundles $(\mathcal{V},\nabla)$ of rank $n$ over $C$ with trivial determinant bundle modulo the action of the complex gauge group.
 When we speak about moduli spaces, we are always assuming that the base hyperbolic Riemann surface $C$ is compact, since we do not want to deal here with representation varieties with fixed parabolic weights at the cusps.
We saw that Lyapunov exponents are defined for every flat bundle with non-expanding cusp monodromies. Hence if $C$ is compact, they are defined for every point of $\drham$.  
We can then define  functions, which we will call the Lyapunov exponent functions, from the de Rham moduli space
\[\lambda_i: \drham\to \RR,\quad (\mathcal{V},\nabla)\mapsto \lambda_i(\mathcal{V},\nabla),\quad i=1,\dots,n\]
that send a flat holomorphic bundle over $C$ to its $i$th Lyapunov exponent.
Recall that  we have to take care only of half of the Lyapunov spectrum since it is symmetric (see Remark \ref{rem:symmspec}). These are the invariants that we will consider in the rest of the paper.

\begin{rem}
	Recall that the the De Rham moduli space is biholomorphic via Riemann-Hilbert correspondence to the complex character variety, also known as Betti moduli space, 
	\[\betti:=\Hom(\pi_1(C),\SL_n(\CC))//\SL_n(\CC),\]
	which is the moduli space of reductive representations of the fundamental group of $C$ into $\SL_n(\CC)$. The De Rham moduli space and the character variety are also homeomorphic via Simpson correspondence to the Hitchin moduli space, which is the moduli space  of rank $n$ polystable Higgs bundle   over $C$ with trivial determinat bundle  and with vanishing Chern classes.\\ 
	The Lyapunov exponents functions can be then also considered on these moduli spaces, but they are more naturally defined on $\drham$. 
\end{rem}

\subsection{ Shatz stratification}
\label{sec:drshatz}

Since Lyapunov exponents are related to degrees of holomorphic subbundles of flat bundles, we recall here the existence of a stratification of the de Rham moduli space given by Harder-Narasimhan type, called Shatz stratification. We then describe more in detail the minimal stratum containing the unitary representations and the maximal stratum, which we identify with the oper locus.  

We  recall some basic definitions about Harder-Narasimhan filtrations and maximal degree subbundles following \cite{H-N}.
	Let $\mathcal{V}$ be a holomorphic vector bundle over $C$ of rank $n$. Recall that the \emph{degree $\deg(\mathcal{V})$ of $\mathcal{V}$} is the first Chern class of $\mathcal{V}$, or equivalently the degree of the determinant bundle $\det(\mathcal{V})$.
	The  \emph{slope} $\mu(\mathcal{V}) $ of $\mathcal{V}$ is defined as  the quotient of the degree and the rank \[\mu(\mathcal{V})=\deg(\mathcal{V})/\rk(\mathcal{V}).\]

Notice that both degree and rank are additive functors, while the slope is not.

\begin{defn}
	A vector bundle $\mathcal{V}$ is \emph{(semi)stable} if, for every  holomorphic subbundle $\mathcal{E}\subset \mathcal{V}$, it holds $\mu(\mathcal{E})<(\leq) \mu(\mathcal{V})$, or equivalently $\mu(\mathcal{V})<(\leq) \mu (\mathcal{V}/\mathcal{E})$.
\end{defn}

We want now to define the Harder-Narasimhan filtration of a vector bundle.
If $\mathcal{V}$ is not semistable, we say that $\mathcal{E}\subset \mathcal{V}$ is maximal if  $\mathcal{E}$ is semistable and for every $\mathcal{E'}$ such that $\mathcal{E}\subsetneq \mathcal{E'}\subset \mathcal{V}$, it holds $\mu(\mathcal{E})>\mu(\mathcal{E'})$. In other words, $\mathcal{E}$ is maximal if it is the semistable subbundle with maximal slope. One can show that it exists and it is unique.
One can moreover show that $\mathcal{E}$ is maximal if and only if  $\mathcal{E}$ is semistable and for every $\mathcal{Q}\subset \mathcal{V}/\mathcal{E}$, it holds $\mu(\mathcal{Q})<\mu(\mathcal{E})$.

\begin{defn}
	The \emph{Harder-Narashiman filtration} of a holomorphic vector bundle $\mathcal{V}$ of rank $n$  is a filtration by holomorphic  subbundles
	\[0=\mathcal{V}_0\subset \mathcal{V}_1\subset \dots \subset \mathcal{V}_l=\mathcal{V}\]
	such that $\mathcal{V}_i/\mathcal{V}_{i-1}$ is semistable and $\mu(\mathcal{V}_i/\mathcal{V}_{i-1})>\mu(\mathcal{V}_{i+1}/\mathcal{V}_{i})$.
	This filtration always exists and it is unique.
	
	We call the collection
	\[(\mu_1,\dots,\mu_n),\quad \mu_i=\mu(\mathcal{V}_i/\mathcal{V}_{i-1})\] 
	of slopes (possibily repeated depending on the rank of $\mathcal{V}_i/\mathcal{V}_{i-1}$) the \emph{Harder-Narashiman type} of $\mathcal{V}$.  	
\end{defn}

For example, if $\mathcal{V}$ is semistable, the Harder-Narashiman type of $\mathcal{V}$ is simply given by $(\mu_1=\mu(\mathcal{V}),\dots, \mu_n=\mu(\mathcal{V}))$.

\begin{rem}
	The bundles $\mathcal{V}_i$ appearing in the Harder-Narasimhan filtration satisfy $\mu(\mathcal{V}_i)>\mu(\mathcal{V}_{i+1})$.
	Moreover, each one of the following conditions can be substituted to the second condition in the definition of Harder-Narasimhan filtration:
	\begin{enumerate}
		\item $\mu(\mathcal{V}_i/\mathcal{V}_{i-1})>\mu(\mathcal{V}_{i+1}/\mathcal{V}_{i})$.
		\item $\mathcal{V}_i/\mathcal{V}_{i-1}$ is maximal in $\mathcal{V}/\mathcal{V}_{i-1}$.
		\item $\mu_i$ is the minimal slope among the slopes of quotients of $\mathcal{V}_i$.
	\end{enumerate}
	In particular  $\mathcal{V}_1$ is the sub-line bundle with the maximal slope among all subbundles of $\mathcal{V}$. For the proof of the previous statements see \cite{H-N}.
\end{rem}



We want to state a central theorem  about the upper-semicontinuity of the Harder-Narashiman type.
There is a partial ordering on vectors given by Harder-Narashiman types, namely
\[(\mu_1,\dots,\mu_n)\leq (\nu_1,\dots,\nu_n) \Longleftrightarrow \sum_{i=1}^k \mu_i\leq \sum_{i=1}^k \nu_i \text{ for all } k=1,\dots, n.\]
One can also visualize this partial ordering by drawing a semi-polygon in the plane that has vertices with coordinates $(\rk(\mathcal{V}_i),\deg(\mathcal{V}_i))$. The partial ordering is then given by checking if a semi-poligon is above an other one.

Now we state the main theorem by Atiyah and Bott, which was proven  in the general case of higher dimensional base space by Shatz.

\begin{thm}[\cite{atiyah},\cite{shatz}]
	The Harder-Narasimhan type defines an  upper semicontinuous function, meaning that if $\mathcal{C}_{\vec{\mu}}$ is the space of vector bundles with Harder-Narasimhan type $\vec{\mu}$, then
	\[\overline{\mathcal{C}_{\vec{\mu}}}\subseteq \bigcup_{\vec{\nu}\geq \vec{\mu}} \mathcal{C}_{\vec{\nu}}.\] In particular there is a stratification of $\mathcal{M}^{(n)}_{DR}$ given by the Harder-Narashiman type called Shatz stratification.
\end{thm}
We will now discuss more in detail the minimal and the maximal Shatz strata.

\subsection{Minimal Shatz stratum}
\label{subsec:minimalstrat}

The minimal stratum is easy to describe, since it is clearly given by the locus of semistable bundles. Note that  the degree of every vector bundle in $\drham$ is zero, since this is the moduli space of flat vector bundles.  Hence  the Harder-Narashiman type of a semistable bundle is  $(\mu_1=0,\dots, \mu_n=0)$. The minimal stratum is an open dense set. The subset of semistable but not stable bundles is the closed  subset of the minimal stratum corresponding to the subset of reducible representations in $\betti$.

%

Thanks to the  theorem by Narasimhan and Seshadri  we know what is the  closed subset of the stable locus in $\drham$ corresponding to the unitary locus of $\betti$.

\begin{thm}[\cite{narashiman-seshandri}]
	\label{thm:n-s}
	The locus of irreducible unitary representations in $\betti$ corresponds to the locus of flat vector bundles $(\mathcal{V},\nabla)\in \drham$ where $\mathcal{V}$ is stable and $\nabla$ is the harmonic metric connection. 
	This subset  is a $(n^2-1)(g-1)$-dimensional closed subvariety of the stable locus of $\drham$. 
\end{thm}

In the Hitchin moduli space the unitary locus  corresponds to the locus of Higgs bundles with zero Higgs field.

\subsection{Maximal Shatz stratum: oper locus}

We want to describe the maximal locus of the Shatz stratification, namely the locus where the flat vector bundles have maximal Harder-Narashiman type. 
In order to do this, we have to introduce the notion of opers and state their main properties. We will follow the survey \cite{wentworthsurvey}.

\begin{defn}
	A \textit{$\SL_n$-oper} is a rank $n$ holomorphic vector bundle $\mathcal{V}$ with trivial determinant bundle, equipped with a flat holomorphic connection $\nabla$ and a  filtration by holomorphic subbundles
	\[0=\mathcal{V}_0\subset \mathcal{V}_1\subset \dots \subset \mathcal{V}_n=\mathcal{V}\]
	such that
	\begin{enumerate}
		\item $\nabla(\mathcal{V}_i)\subseteq \mathcal{V}_{i+1}\otimes \mathcal{K}_C$;
		\item $\nabla:\mathcal{V}_i/\mathcal{V}_{i-1}\to \mathcal{V}_{i+1}/\mathcal{V}_i\otimes \mathcal{K}_C$ is an isomorphism.
	\end{enumerate}
\end{defn}

Let us now define the oper locus $\Op_n(C)\subset \drham$ as the subset of the de Rham moduli space of flat vector bundles admitting an oper structure. This definition makes sense since the oper structure, i.e. the oper filtration, is unique for a fixed oper $(\mathcal{V},\nabla)$. 
The uniqueness of the oper structure is a consequence of  the following central proposition.
\begin{prop}[\cite{wentworthsurvey}]
	\label{prop:oper}
	Let $(\mathcal{V},\nabla)$ be a $\SL_n$-oper. The oper structure on $\mathcal{V}$ is uniquely determined by the line bundle 
	\[\mathcal{V}/\mathcal{V}_{n-1}\cong \mathcal{V}_1\otimes \mathcal{K}_C^{-(n-1)}.\]
	Moreover 
	\[(\mathcal{V}/\mathcal{V}_{n-1})^n\cong \mathcal{K}_C^{-n(n-1)/2} \quad\text{ and }\quad \det(\mathcal{V}_j)\cong \mathcal{V}/\mathcal{V}_{n-1}\otimes \mathcal{K}_C^{nj-(j(j+1)/2)}.\]
\end{prop}

In particular the isomorphism class of $\mathcal{V}$ is fixed on every connected component of $\Op_n(C)$ and  each component  parametrizes  the space of holomorphic connections on a fixed holomorphic vector bundle. These components are classified by the choice of the line bundle $\mathcal{V}/\mathcal{V}_{n-1}$ which is defined by the property $(\mathcal{V}/\mathcal{V}_{n-1})^n\cong K_X^{-n(n-1)/2}$. Hence $\Op_n(C)$ has $n^{2g}$ connected components, which also corresponds to the number of ways of lifting a monodromy representation in $\PSL_n(\CC)$ to $\SL_n(\CC)$ (see \cite[Remark 4.2]{wentworthsurvey}). 
Using Proposition \ref{prop:oper} one can also prove that if a holomorphic bundle has the structure of an oper then it must be  an irreducible flat vector bundle, or equivalently the representation that it defines is simple (\cite[Prop. 4.8]{wentworthsurvey}).


We want now to parametrize the oper locus in such a way that it will be easy to see  why in rank $2$  the oper locus corresponds to  the set of holonomies of projective structures inducing the same complex structure  on $C$ (see Proposition \ref{prop:operproj}).

Consider the $n$-th order  differential operator on $\HH$ locally of the form
\begin{equation}
Dy=y^{(n)}+Q_2 y^{(n-2)}+\dots Q_n y
\end{equation}
where $Q_j$ are pull-backs of local sections of $\mathcal{K}_C^j$.
This differential operator induces a short exact sequence of $\underline{\CC}$-modules
\begin{equation}
\label{eq:exactseqdiffeq}
0\to \mathbb{V}\overset{\varphi}{\longrightarrow} \mathcal{K}_C^{1-(n+1/2)}\overset{D}{\longrightarrow} \mathcal{K}_C^{n+1/2}\to 0
\end{equation}
and we say that the local system $\mathbb{V}$ is realized in $\mathcal{K}_C^{1-(n+1/2)}$.
Clearly the space of all such local systems is parametrized by the affine space modeled on $\bigoplus_{j=2}^n H^0(C,\mathcal{K}_C^j)$.

We can now state another characterization of opers.

\begin{prop}[\cite{wentworthsurvey}]
	\label{prop:operdiffeq}
	Let $(\mathcal{V},\nabla)$ be a flat vector bundle. Then $(\mathcal{V},\nabla)$ is an oper if and only if its associated local system is realized in $\mathcal{K}_C^{1-(n+1/2)}$.\\
	This characterization defines an isomorphism between each connected component of $\Op_n(C)$ and the affine space modeled on the Hitchin base $\bigoplus_{j=2}^n H^0(C,\mathcal{K}_C^j)$. It follows also that the dimension of $\Op_n(C)$ is $(n^2-1)(g-1)$.
\end{prop}

We give now the explicit map between local systems realized in  $\mathcal{K}_C^{1-(n+1/2)}$ and opers realizing the correspondence of the proposition above.
Assume  that we are given the exact sequence \eqref{eq:exactseqdiffeq} and let $\mathcal{V}:=\mathcal{O}_C \otimes_{\underline{\CC}} \mathbb{V} $ be the associated flat vector bundle.  Define the oper filtration by
\begin{equation}
\label{eq:operfromdiffeq}
\mathcal{V}_{n-k}:=\{\sum_{i=1}^n f_i\otimes v_i\colon \sum_{i=1}^n f_i^{(j)}\varphi(v_i)=0,\ j=0,\dots,k-1\}\quad \text{ for } k=1,\dots,n-1,
\end{equation}
where $f_i^{(j)}$ is the $j$-th derivative of the local holomorphic function $f_i$.
We will see that the construction \eqref{eq:operfromdiffeq} of the oper filtration is useful to relate opers and projective structures in rank 2.

Let us  finally recall the theorem   stating that the oper locus is the maximal Shatz stratum.

\begin{thm}[\cite{wentworthsurvey}]
	\label{thm:opermax}
	The maximal stratum of $\mathcal{M}_{DR}^{(n)}$ is the oper locus $\Op_n(C)$. The Harder-Narashiman filtration of an oper is the oper filtration itself and the Harder-Narashiman type is given by
	\[\mu_i=\mu(\mathcal{K}_C^{(n+1)/2-i})=(n+1-2i)(g-1).\]
	By upper-semicontinuity, the oper locus is a closed embedded subset of $\mathcal{M}_{DR}^{(n)}$.
\end{thm}

Using a result of  \cite{simpsonlimits} it is easy to show that there is  only one variation of Hodge structures on an each connected component of the oper locus.

\begin{prop}
	\label{prop:opervhs}
	The only variation of Hodge structures on an each connected component of the oper locus is given by the $(n-1)$-th symmetric power of the maximal Higgs one in rank $2$. The maximal Higgs variation of Hodge structures in rank $2$ corresponds to the uniformizing representation of $C$.
\end{prop}

\begin{proof}
	Let $(\mathcal{V},\nabla,\{\mathcal{V}_{i}\}_{i=0,\dots,n})$  be an oper.
	By Proposition \ref{prop:oper} the holomorphic vector bundle $\mathcal{V}$ is fixed on each connected component of the oper locus. Moreover the flat bundle $(\mathcal{V},\nabla)$ is irreducible. 
	Let now 
	\[F^{k+1}=0\subset \dots \subset F^0=\mathcal{V}\]
	be the filtration associated to a VHS structure on $\mathcal{V}$. Since $(\mathcal{V},\nabla)$ is irreducible, the variation of Hodge structure $(\mathcal{V},\nabla,\{F^{i}\}_{i=0,\dots,k+1})$ is irreducible as VHS. Hence the corresponding system of Hodge bundles is stable.
	Since both the oper and the VHS filtrations are semistable, by \cite[Prop. 4.3]{simpsonlimits} they are the same and the associated Higgs bundle has to be the maximal Higgs system of Hodge bundles. This system of Hodge bundles is clearly given by the symmetric power of the maximal Higgs in rank 2. 
	
	In order to check that the rank 2 maximal Higgs variation of Hodge structure defines the uniformizing representation of $C$, notice that the Higgs field can be identified with the derivative of the period map $p:\HH\to \HH$. Since the Riemann surface $C$ is compact, the period map is proper. Hence if the Higgs field is an isomorphism, the period map is a covering map, since it is a proper local isomorphism. It then has to be an isomorphism since $\HH$ is simply connected. Hence the period map induces an isomorphism between $C\cong \HH/\pi_1(C)$ and $\HH/\rho(\pi_1(C))$, where $\rho$ is the corresponding representation. It follows that $\rho$ is the uniformizing representation of $C$.
\end{proof}

\subsection{Lyapunov exponents on Shatz strata}
\label{sec:shatzbound}

The Harder-Narasimhan type provide then a natural bound for the sum of Lyapunov exponents thanks to Theorem \ref{thm:ekmz}. It is natural then to consider the Lyapunov exponents functions with respect to the Shatz strata. We also show that Lyapunov exponents functions are unbounded on the maximal stratum.

The bound on the Lyapunov exponent functions on the minimal stratum defined by semistable bundles is trivial, since there the Harder-Narasimhan filtration is trivial. Notice that all  Lyapunov exponent functions restricted to the closed subset of this stratum given by  the unitary locus are zeros.

On the maximal stratum, which by Theorem \ref{thm:opermax} is the oper locus $\Op_n(C)$, we have the maximal possible bound of the Lyapunov exponent functions. 
We can actually compute the bound, since we know the Harder-Narashiman type of an oper.

\begin{prop}
	\label{prop:lyapoper}
	If $\mathcal{V}\in \Op_n(C)$ is in the oper locus then 
	\[\sum_{i=1}^k \lambda_i(\mathcal{V})\geq k(n-k),\quad k=1,\dots,n.\]
	The above inequalities are sharp, since they are achieved in the special point corresponding to the only flat vector bundle of the maximal Shatz stratum  underlying a variation of Hodge structures. This special point  corresponds to the  $(n-1)$-th symmetric power of the uniformizing representation of $C$.
\end{prop}

\begin{proof}
	Using  Theorem \ref{thm:ekmz} we see that the Lyapunov spectrum dominates the Harder-Narasimhan type. 
	By Theorem \ref{thm:opermax} we know that the Harder-Narasimhan filtration of an oper is the oper filtration itself. Using  Proposition \ref{prop:oper} and arguing inductively we can compute the slopes of the subbundles $\mathcal{V}_k$ defining the  oper filtration. We see then that 
	\[\deg(\mathcal{V}_k)=k(n-k).\]
	To prove equality in the only variation of Hodge structures point (see Proposition \ref{prop:opervhs}) notice that  in the rank $2$ case the uniformizing variation of Hodge strucure is of weight one. Hence by Theorem \ref{thm:ekz} we can compute the top Lyapunov exponent, which is $1$. Now it is enough to compute how Lyapunov exponents change under symmetric power and check that we get indeed the equality.
\end{proof}

Now that we have proved a lower bound for the sum of Lyapunov exponents of Shatz strata, it is natural to ask if the Lyapunov exponent functions are unbounded on these strata. Using the recent result of Dujardin and Favre \cite[Th. A]{dujfavre} about the growth of Lyapunov exponents for meromorphic families, we are able to show the following result for the maximal stratum.

\begin{thm}
	\label{thm:shatzunbounded}
	The top Lyapunov exponent function is unbounded on the maximal Shatz stratum, the oper locus, with logarithmic growth near the boundary of the character variety.
\end{thm}
\begin{proof}
	First of all note that the character variety and the de Rham moduli space are biholomorphic, hence by Theorem \ref{thm:opermax} the oper locus, being the maximal Shatz stratum, is a closed embedded subset of the character variety.
	Now recall that by Proposition \ref{prop:operdiffeq}  each connected component of the oper locus in rank $n$ is biholomorphic to the Hitchin base $\bigoplus_{j=2}^n \Hcoh^0(C;\mathcal{K}_C^j)$. \\
	By the result of \cite{dujfavre}, a meromorphic family from the unit disk to the space of non-elementary representations in the character variety in rank $2$ which is holomorphic outside of zero and cannot be holomorphically extended in 0 yields a logarithmic growth of the Lyapunov exponent near zero. A connected component of the oper locus in rank 2 is given by the embedded vector space $\Hcoh^0(C;\mathcal{K}_C^2)$ in the character variety and does not intersect the space of elementary representations. Since on a complex vector space there are a lot of meromorphic maps from the disk $\mathbb{D}\to \Hcoh^0(C;\mathcal{K}_C^2)$ holomorphic in $\mathbb{D}\setminus\{0\}$ and not holomorphic in zero, we have shown that the Lyapunov exponent function is unbounded in rank 2. To generalize the result to any rank it is enough to recall that, by Proposition \ref{prop:oper},  each component of the oper locus parametrizes  the space of holomorphic connections on a fixed holomorphic vector bundle. Since the rank $n$ oper locus $\Op_n(C)$ contains the $(n-1)$-symmetric power of the uniformizing representation, it is clear that the $(n-1)$-symmetric power of any representation in $\Op_2(C)$ is contained in $\Op_n(C)$. Indeed the symmetric power of a representation in  $\Op_2(C)$ defines a holomorphic connection on the  holomorphic vector bundle given as the symmetric power of the vector bundle underlying the uniformizing representation. Since the top Lyapunov exponent of the $n$-symmetric power of a representation $\rho$ is $n\times \lambda_1(\rho)$, the result  in rank 2 implies the genral result for any rank. 
\end{proof}

	
\section{Special loci of Betti and de Rham moduli spaces in rank two}
\label{sec:rank2}

In this section we specialize to the study of the moduli space of rank two flat bundles. In this special case we can focus between the relation of special loci in the Betti moduli space and special loci in the De Rham moduli space. In the next section we will study properties of the Lyapunov exponent function on these loci. In particular we  identify the oper locus in $\drhamtwo$ with the space of projective structures underlying the same Riemann surface structure $C$ . This will lead in the next section to the generalization of the main result of \cite{deroin}. The material of this section is well-known but scattered throughout the literature. We present a summarizing picture collecting many different results.

\subsection{Special loci of the Betti moduli space}

Recall that the Betti moduli space
	\[\bettitwo:=\Hom(\pi_1(C),\SL_2(\CC))//\SL_2(\CC)\]
 is the moduli space of reductive representations of the fundamental group of $C$ into $\SL_2(\CC)$. 
 A discrete subgroup $\Gamma\subset \PSL_2(\CC)$ is called \emph{Kleinian}. A \emph{quasi-Fuchsian group} is a Kleinian group $\Gamma$ such that the accumulation points of the $\Gamma$-action on $\partial \HH^3$ is a quasi-circle. They are quasi-conformal deformations of Fuchsian groups, namely  discrete subgroups of $ \PSL_2(\RR)$.
These special subgroups of $\SL_2(\CC)$ define special subsets of $\mathcal{M}^{(2)}_B$. We recall now their properties.
Let $\qfuchs(S)$ be the subspace of quasi-Fuchsian representations and $ \mathcal{D}(S)$ be the subspace of discrete representations. 
The subset of discrete representations $ \mathcal{D}(S)\subset \bettitwo$ is closed in the analytic topology. Its interior is the locus $\qfuchs(S)$ of quasi-Fuchsian representations. 


Recall also that a representation is called \emph{elementary} if its action on $\HH^3$ fixes a point or an ideal point, or if it preserves an unoriented geodesic. Otherwise it is called  \emph{non-elementary}. Equivalently a representation is elementary if and only if it is unitary or reducible or if the image is conjugated to a subgroup of the group generated by 
$\langle \left(\begin{smallmatrix} \lambda & 0  \\0 & \lambda^{-1}  \end{smallmatrix} \right),\left(\begin{smallmatrix}  0 &-1 \\1 & 0  \end{smallmatrix} \right)\rangle$.
We denote by ${\bettitwo}'$  the subspace of non-elementary representations. It is a Zariski-dense subset contained in the smooth locus of $\bettitwo$. 

\subsubsection{Real Betti moduli space}
We can also consider the subset of representations into $\SL_2(\RR)$ as a subset of the Bett moduli space.
et the real Betti moduli space defined as 
\[{\mathcal{M}_{B,\RR}^{(2)}}:=Hom(\pi_1(S),\SL_2(\RR))//\SL_2(\RR).\]

\begin{thm}[\cite{goldmantopcomp}]
	The real representation variety $\mathcal{M}_{B,\RR}^{(2)}$ has one connected component for each even integer $e$ with $0\leq |e| \leq 2g-2$. The integer $e$ correponds to the Toledo invariant or the Euler number associated to a representation.\\
	In the case of the maximal integer $e=2g-2$, the connected component ${\mathcal{M}_{B,\RR,2g-2}^{(2)}}$ is the same as the space of Fuchsian representations. \\
	In the case of the minimal integer $e=0$, the connected component ${\mathcal{M}_{B,\RR,0}^{(2)}}$ is contained in the space of elementary representations.
\end{thm}

\subsubsection{Projective structures and holonomy map}
\label{sec:projstr}
Finally we want to recall the definition of projective structures and the properties of the holonomy map from the moduli space of projective structures to the Betti moduli space.
A \emph{complex projective strucuture} on $S$ is a maximal atlas of charts mapping open sets of $S$ into $\PP_{\CC}^1$ such that the transition functions are restricition of M\"obius tranformations. 
Equivalently, a projective structure on $S$ is given by a pair $(\dev,\rho)$ where  $\rho:\pi_1(C)\to \PSL_2(\CC)$ is a representation called \emph{holonomy representation} and $\dev:\HH\to \PP^1_{\CC}$ is a $\rho$-equivariant immersion called \emph{developing map}.
The pair has to be considered modulo the natural equivalence relation given by precomposition of developing maps with orientation-preserving diffeomorphisms of $S$ homotopic to the identity on one side and by conjugation of $\PSL_2(\CC)$ on the other.

Let $\Proj(S)$ be space of projective structures on $S$ and $\Teich(S)$ be the Teichm\"uller space.
Note that since M\"obius transformations are holomorphic, a projective structure also determins a complex structure. We can then consider the forgetful map
\[p: \Proj(S)\to \Teichm(S).\]
Denote by $\Proj(C):=p^{-1}(C)$ the fiber over a Riemann surface $C$, which is the space of projective structures inducing the same holomorphic structure.
\begin{rem}
	\label{rem:projstr}
	The space $\Proj(C)$ can be identified with the space of quadratic differentials $H^0(C,\mathcal{K}_C^2)$, using the Schwarzian derivative.
	The identification is given by associating to $(\dev,\rho)\in \Proj(C)$ the push-forward to $C$ of the quadratic differential $S(\dev)$ on $\HH$, where $S(\dev)$ is the Schwarzian derivative. In the other direction, one associates to $\phi\in H^0(C,\mathcal{K}_C^2)$ the projective structure given by $(u_1(z)/u_2(z), \rho)$ where $u_1,u_2$ is a basis of solution of the differential equation
	\begin{equation}
	\label{eq:diffeqranktwo}
	u''(z)+\frac{1}{2}\tilde{\phi}(z)u(z)=0.
	\end{equation}
	Here $\tilde{\phi}(z)\dd z^2$ is the pull-back of $\phi$ to the universal cover $\HH$ and $\rho$ is the monodromy associated to the differential equation.
\end{rem}

We can relate the set of projective structures and the Betti moduli space via the holonomy map, which sends a projective structure to its associated holonomy representation. 
Recall that, by a theorem of Gallo, Kapovich and Marden (\cite{gallokapovichmarden}) the holonomy map 
	\[\hol : \Proj(S)\to \bettitwo\]
	has image in ${\bettitwo}'$, it is surjective on this set and it is a local biholomorphism.
	
	We will now restrict the holonomy map to the fibers $\Proj(C)$ of the forgetful map $p:\Proj(S)\to \Teichm(S)$ and recall the following result.
	
	\begin{thm}[\cite{dumassurvey}]
		For every $C\in \Teich(S)$, the restriction $\hol_{|\Proj(C)}$ is a proper holomorphic embedding. Consequently the image $\hol(\Proj(C))$ is a complex-analytic subvariety of
		${\mathcal{M}^{(2)}_B}'$.
	\end{thm}
	Notice that, since $\Proj(C)$ is an affine space modeled on $H^0(C,\mathcal{K}_C^2)\cong \CC^{3g-3}$, then also the image  $\hol(\Proj(C))$ is.
	We will later show  how $\hol(\Proj(C))$ is the same as the oper locus $\Op_2(C)\subset \drhamtwo$.
	We now describe more in detail the intersection of the the quasi-Fuchsian, the discrete and the Fuchsian loci and $\hol(\Proj(C))$.

\begin{thm}[\cite{faltings},\cite{dumassurvey}]
	For each $C\in \Teich(S)$, the intersection of $\hol(\Proj(C))$ and ${\mathcal{M}_{\text{B},\RR}^{(2)}}$ is transversal. Moreover, 
	the intersection of $\hol(\Proj(C))$ and the Fuchsian locus ${\mathcal{M}_{\text{B},\RR,2g-2}^{(2)}}$ is countable. Each one of these Fuchsian points is contained in the open set of quasi-Fuchsian holonomy representations  in $\hol(\Proj(C))$. These open sets are  connected, contractible and biholomorphic to $\Teich(S)$. The closure of these open sets gives the space of discrete representations in $\hol(\Proj(C))$.
\end{thm}

\begin{rem}
	\label{rem:bers}
	There is a special Fuchsian representation point in $\hol(\Proj(C))$, namely the point corresponding the uniformizing representation of $C$. The open connected set $B(C)$ of quasi-Fuchsian representations containing this point is given by the image under the holonomy map  of the Bers embedding \[B(C)=\hol(Im(\Teich(S)\hookrightarrow \Proj(C)))\subset \bettitwo.\]
	The Bers embedding is defined by 
	\[\Teich(S)\hookrightarrow \Proj(C), \quad Y\mapsto \Sigma_Y(C)\]
	where $\Sigma_Y(C)$ is the projective structure on $C$ induced by the quasi-Fuchsian group $Q(C,Y)$ given by the simultaneous uniformization theorem. 
\end{rem}


%


\subsection{Special loci of de Rham moduli space}
We now specialize to subsets of the de Rham moduli space in rank two. First we describe the variations of Hodge structures loci and then investigate the Shatz stratification in rank two. In  particular we show that the maximal stratum is the same as the set $\hol_{|\Proj(C)}$  of holonomies of projective structures underlying the same Riemann surface structure treated above.

\subsubsection{Variation of Hodge structures locus in rank 2}
\label{subsec:vhsrank2}

Thanks to Hitchin in \cite{hitchin}, we know how the connected components $P_e\subset \drhamtwo$ of the variation of Hodge structures locus look like. More in detail, we know the corresponing system of Hodge bundles. These are indexed by an integer $0\leq e\leq g-1$. For $e=0$, we already recalled that $P_0$ is the space of variations of Hodge structures of weight 0, which corresponds to unitary representations.
For $e>0$, the space $P_e$ parametrizes Higgs bundles of the form 
\[\mathcal{V}=\mathcal{V}^0\oplus \mathcal{V}^1,\quad \Phi:\mathcal{V}^1\to \mathcal{V}^0\otimes \mathcal{K}_C,\]
where $\mathcal{V}^0$ and $\mathcal{V}^1$ are line bundles of degrees  $-e$ and $e$ respectively. These systems of Hodge bundles correspond to weight one complex variations of Hodge structures. By the trivial determinant condition $\det(\mathcal{V}^0\oplus \mathcal{V}^1)=\mathcal{V}^0\otimes \mathcal{V}^1\cong \mathcal{O}_C$, and so $\mathcal{V}^0\cong (\mathcal{V}^1)^*$.
Note that the Higgs field $\Phi$ is a section of the line bundle $(\mathcal{V}^1)^*\otimes \mathcal{V}^0\otimes \mathcal{K}_C$ of degree $2g-2-2e$, hence $e\leq g-1$ and in case of equality, $\Phi$ is an isomorphism.
Let $D\in \Symm^{2g-2-2e}(C)$ be the divisor of $\Phi$.
Then we have only finitely many possibilities for $\mathcal{V}^1$,  determined by the isomorphism  $(\mathcal{V}^1)^2\cong \mathcal{K}_C\otimes \mathcal{O}_C(-D)$.
We have then the parametrization
\[P_e\cong   \Symm^{2g-2-2e}(C)\times \text{finite} \# \]
which gives that the dimension of $P_e$ is $2g-2-2e$.
For $e=g-1$, the Higgs field is an isomorphism and we get the only variation of Hodge structure on the oper locus, which corresponds to the uniformizing representation of $C$ (see Proposition \ref{prop:opervhs}).

Note finally that the locus $P_e$ belongs to the $e$-th Shatz stratum, which is defined as the locus of flat bundles such that the degree of the maximal destibilizing subline bundle is $e$, for $0\leq e\leq g-1$. Indeed, by Proposition \ref{prop:w1k3HN} the first piece of the Hodge filtration $\mathcal{V}^1$ is the maximal destabilizing subsheaf of $\mathcal{V}$.

\subsubsection{Shatz stratification}

In rank two, the Shatz stratification consists of  $g$ Shatz strata. The degree of the maximal destabilizing subbundle varies in the set $\{0,1,\dots, g-1\}$. Indeed flat bundles in the oper locus, which is the maximal stratum, have maximal destabilizing subbundle of degree $g-1$ (see Theorem \ref{thm:opermax}). 

The minimal Shatz stratum  is given by semistable flat bundles. This is an open dense subset containing the space of unitary representations as a closed subspace. The space of unitary representations is the same as the locus $P_0$ of variations  of Hodge structures of  weight zero. 
The subspace of the semistable locus given by semistable but not stable bundle is a closed subset and corresponds to non simple representations. This subspace is contained in the space of elementary representations.

The maximal Shatz stratum can be identified with the oper locus  $\Op_2(C)\subset \drhamtwo$.  It is a $(3g-3)$-dimensional closed subvariety given by the locus of flat bundles  having a sub-line bundle with maximal possible degree, namely $g-1$. Indeed in this case the oper filtration is given by $0\subset \mathcal{V}_1\subset \mathcal{V}$ with $\mathcal{V}_1\cong \mathcal{K}_C^{1/2}$ (see Proposition \ref{prop:oper}).
We  describe the correspondence   between the oper locus and set of holonomies of the projective structures  inducing the same complex structure $C$.

\begin{prop}
	\label{prop:operproj}
	The subset $\hol(\Proj(C))\subset \betti$ corresponds via the Riemann-Hilbert correspondence to the oper locus $\Op_2(C)\subset \drhamtwo$. Moreover, if $\mathcal{V}$ is an oper, the meromorphic  map 
	\[s_{\mathcal{V}_1}:\HH\to \PP^1\]
	defined by the inclusion of the sub-line bundle $\mathcal{V}_1\subset \mathcal{V}$ given by the oper filtration is  the developing map associated to the projective structure corresponding to the oper.
\end{prop}
\begin{proof}
	For the proof of the first statement, it is enough to combine the description of opers given by Proposition \ref{prop:operdiffeq} and  Remark \ref{rem:projstr} about the relation of projective structures and differential equations.
	
	In order to prove the second statement we need to use  the construction of the oper filtration \eqref{eq:operfromdiffeq} starting from the local system defined by the differential equation. Using the notation of as in \eqref{eq:operfromdiffeq}, the inclusion $\mathcal{V}_1\cong \mathcal{K}_C^{1/2}\subset \mathcal{V}=\mathcal{O}_C\otimes \mathbb{V}$ is given by
	\[	\mathcal{V}_{1} =\{ f_1\otimes v_1+f_2\otimes v_2\colon  f_1\varphi(v_1)+f_2\varphi(v_2)=0\}\subset \mathcal{V} \]
	where $v_1,v_2$ are a local basis of the solution of the local system.
	Recall that the devoloping map in Remark \ref{rem:projstr} was exactly  defined as the meromorphic function $ \dfrac{\varphi(v_1)}{\varphi(v_2)}$. 
	It is immediate to see that the map $s_{\mathcal{V}_1}$ defined by the inclusion of the pull-back of $\mathcal{V}_1$ into the trivial vector bundle on $\HH$ is given by the same meromorphic map.
\end{proof}


\subsection{Summarizing picture}

We present a summarizing picture representing the special loci that we described. In each one of  the $g$ Shatz strata we find the complex variation of Hodge structures Loci $P_e$. The locus $P_0$ of weight zero variation of Hodge structures corresponds to the locus of unitary representations and it is contained in the minimal open Shatz stratum of semistable flat bundles. The locus $P_{g-1}$ of maximal Higgs variation of Hodge structures, corresponding to the $s^{2g}$ lifts to $SL_2(\CC)$ of the uniformizing representation of $C$, is contained in the oper locus, the maximal Shatz stratum. Using the interpretation available in rank $2$ of representations as holonomies of projective structures, we can understand better the oper locus, since it is the same as the set $\Proj(C)$ of projective structures on $C$ inducing the same original complex structure. Finally the blue locus is the maximal real representation locus of Fuchsian representations.

\begin{figure}[h]
	\centering
	\includegraphics[scale=0.55]{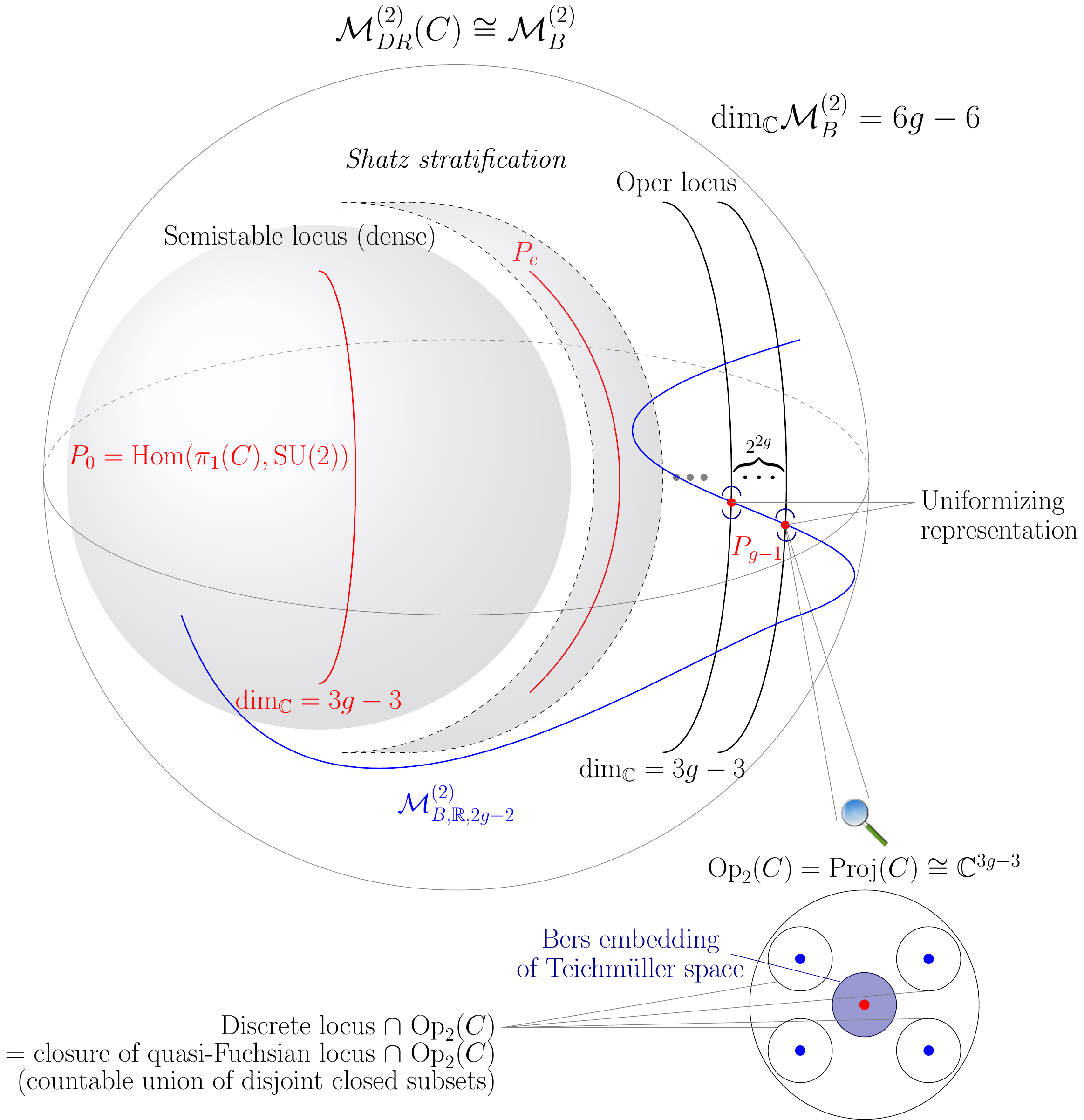}
	\caption{Betti and de Rham moduli spaces in rank 2 and special loci.}
	\label{fig:picture}
\end{figure}

\section{Lyapunov exponent function in rank two}
\label{sec:lyapexpfctrank2}

We want to describe the Lyapunov exponents invariants on the rank two de Rham moduli space, in particular describe their behavior of the special loci described in Figure \ref{fig:picture}. By restricting to the oper locus, we retrieve the main result of \cite{deroin} and we generalize it to other Shatz strata using the bound of  Proposition \ref{prop:lyapoper}, which is an equality in the compact case thanks to Theorem \ref{thm:conjcompact}. This generalization is equivalent to the generalization given in \cite[Theorem F]{deroin} in the context of branched projective structures. 
Notice that in rank two we care only about the top Lyapunov exponent because of the symmetry of the Lyapunov spectrum.   Using the relation of the geodesic flow and the random product on the fundamental group, we reprove known results about continuity of the top Lyapunov exponent function and the characterization of the locus of zero exponent (see \cite{deroin} for the proof in the Brownian motion setup). 

\subsection{Geodesic flow and random products}

Before describing some properties of the top Lyapunov exponent function, we need a lemma that relates the cocycle $G_t:\mathcal{V}_{\rho}\to \mathcal{V}_{\rho}$ defined by parallel transport over the geodesic flow  to the random product of matrices in the monodromy group. We will argue in a similar way as in \cite[Proof of Th. 1]{eskinmatheus}.  By a result of H. Furstenberg \cite{furstenberg}, there is a probability measure $\nu$ on the uniformizing group $\Gamma<\SL_2(\RR)$ of $C$ with support equal to $\Gamma$ such that the Poisson boundary of $(\Gamma,\nu)$ is $(\SO(2,\RR), \text{Leb})$. We will denote by $\lambda_{\rho(\nu)}(u)$ the Lyapunov exponents of $u\in \mathcal{V}_c$ with respect to the random walk of law $\rho(\nu)$ on the monodromy group $\rho(\pi_1(C,c))=\rho(\Gamma)$. By definition, for $\nu^{\NN}$-almost any $(\gamma_1,\dots,\gamma_n,\dots)\in \Gamma^{\NN}$ and any norm $||\cdot||$ on the vector space $\mathcal{V}_c$ it holds
\[\lambda_{\rho(\nu)}(u)=\lim_{n\to \infty} \frac{1}{n}||\rho(\gamma_n)\cdots\rho(\gamma_1)u||.\]

\begin{lem}
	\label{lem:randomwalk}
The Lyapunov exponents $\lambda(\mathcal{V}_{\rho})$ defined by the cocycle given by parallel transport over the geodesic flow of $T^1C$ and the  the Lyapunov exponents $\lambda_{\rho(\nu)}(u)$ given by the random walk on the monodromy group $\rho(\pi_1(C,c))$ coincide.
\end{lem}

\begin{proof}
Recall that by definition the Lyapunov exponent $\lambda(\mathcal{V}_{\rho})$ is defined to be
\[\lambda(\mathcal{V}_{\rho})=\lim_{t\to \infty} \frac{1}{t}\log ||G_t(u)||\]
for almost any $c\in T^1C$ and almost  all $u\in \mathcal{V}_c$. Here $||\cdot||$ is the constant norm which is integrable by Theorem \ref{thm:integrable}.
Since we are dealing with the case of a compact Riemann surface $C$, we can prove a stronger property than integrability of this norm. Consider the lift of the geodesic flow to $T^1\HH$. Since $C$ is compact, a geodesic segment of unit length can cross the boundary of a fundamental domain only a finite number of times. This number is uniformly bounded for all starting points. Hence, there is a constant $M>0$ such that 
\begin{equation}
\label{eq:boundnorm}
||\log(G_t(u))||\leq M\cdot t
\end{equation}
for all $u\in \mathcal{V}_{\rho}$ and all $t\in \RR$. \\
Let us denote by $(z,\theta)\in T^1\HH$ a lift of a point $c\in  T^1C$.
It is well known that a typical trajectory of the random walk in $\Gamma\cdot i\subset \HH$ tracks a geodesic ray in $\HH$ up to sublinear error (see \cite[Lemma 4.1]{chaikaeskin}). This means that for almost all $(\gamma_1,\dots,\gamma_n,\dots)\in \Gamma^{\NN}$ there exists a geodesic ray $\{g_t(z,\theta)\colon t\in \RR\}\subset \HH$  such that
\begin{equation}
\label{eq:tracking}
\text{dist}_{\text{hyp}}(\gamma_n\cdots \gamma_1\cdot i,g_n(z,\theta))=o(n)
\end{equation} 
for almost any $(z,\theta)\in T^1\HH$.\\
Putting together the bound \eqref{eq:boundnorm} and the tracking property \eqref{eq:tracking} we get 

\begin{align*}
\log\left(\frac{||G_n(u)||}{||\rho(\gamma_n)\cdots\rho(\gamma_1)u||}\right)\leq M\cdot \text{dist}_{\text{hyp}}(\gamma_n\cdots \gamma_1\cdot i,g_n(z,\theta))= o(n)
\end{align*}
for any $u\in \mathcal{V}_{\rho}$ in the fiber over $(z,\theta)$. 
We conclude with  the desired result
\begin{align*}
\lambda(\mathcal{V}_{\rho})-\lambda_{\rho(\nu)}(u)=\lim_{n\to \infty}\frac{1}{n}\log\left(\frac{||G_n(u)||}{||\rho(\gamma_n)\cdots\rho(\gamma_1)u||}\right)\leq \lim_{n\to \infty}\frac{1}{n} o(n)=0
\end{align*}

\end{proof}

Thanks to the last lemma we can  prove some properties of the top Lyapunov exponent function using known results about random products of matrices. The next two proposition were already proven in \cite{deroin} in the context of Brownian motion using the same core arguments about random walks. We provide an alternative proof relating the geodesic flow to the random walk instead of   relating the Brownian motion to the random walk. The two point of views lead to the same result since Brownian motions tracks sublinearly the geodesic flow.

\subsection{Continuity of the top Lyapunov exponent function}

Using ones again Lemma \ref{lem:randomwalk}, we can use known results about random walks to prove continuity of the top Lyapunov exponent function.

\begin{prop}
	The top Lyapunov exponent function
	\[\lambda_1:\bettitwo\to \RR_{\geq 0}\]
	is a continuous function. Moreover it is locally Holder continuous and pluri-subharmonic on the set of non-elementary representations.
\end{prop}
\begin{proof}
	By Lemma \ref{lem:randomwalk}, we can use  continuity results of Lyapunov exponents proved  in the context of random walks. By \cite{lepage} the top Lyapunov exponent function is locally Holder continuous on the set of non-elementary representations. Moreover, since the set of elementary representations coincides with the set of zero exponents, the top Lyapunov exponent function is continuous also at these points since all exponents are equal (see for example \cite[Corollary 9.3]{vianabook}). The pluri-subharmonicity follows from \cite[Th. 1.1.]{demarco} or more directly from \cite[Th. 3.7]{dd1}.
\end{proof}

\subsection{Locus of zero top Lyapunov exponent}

We can describe locus where the top Lyapunov exponent vanishes.

\begin{prop}
	\label{prop:zero}
	In rank 2, the Lyapunov exponent associated to a representation $\rho$ is zero if and only if   $\rho$ is elementary. 
\end{prop}
\begin{proof}
	By Lemma \ref{lem:randomwalk}, we can use classical results about random walks in the monodromy group to establish if the top Lyapunov exponent vanishes. 
	By a Theorem of Furstenberg (see \cite[Th. 6.11]{vianabook}), the Lyapunov exponents associated to a random walk in rank 2 are  non zeroes if and only if the the 	cocycle is pinching and twisting. 
	
	If the cocycle is non-pinching, then the monodromy group is contained in a compact subgroup of $\SL_2(\CC)$, so it is the unitary case.
	If the cocycle is non-twisting then the monodromy group is a diagonal subgroup or a triangular subgroup or the image is contained in the subgroup generated by 
	$\langle \left(\begin{smallmatrix} \lambda & 0  \\0 & \lambda^{-1}  \end{smallmatrix} \right),\left(\begin{smallmatrix}  0 &-1 \\1 & 0  \end{smallmatrix} \right)\rangle$.   These cases are exactly the cases of non-elementary representations.
\end{proof}

We will now describe special loci where we can say something about the Lyapunov exponent. The main idea is to use Theorem \ref{thm:lyapexp} to get a lower bound for the top Lyapunov exponent on  the Shatz strata. We can then describe the special loci where we have more information about the Lyapunov exponent. These special loci are the ones described in Picture  \ref{fig:picture}.



\subsection{Lyapunov exponents on Shatz strata}
\label{subsubsec:rank2deroin}

Recall that the Shatz stratification is defined by Harder-Narasimhan type. In rank 2, there is a stratum for each integer $0\leq e\leq g-1$ and for opers, which define the maximal stratum, the maximal destabilizing subbundle has degree $g-1$. Recall moreover that the holonomy map $ \hol : \Proj(S)\to \bettitwo$ has image equal to the space of non-elementary representations.  

\begin{prop}
	\label{prop:rank2shatz}
	Let $\mathcal{V}$ be a flat vector bundle in the $e$-th Shatz stratum and let $\mathcal{L}\subset \mathcal{V}$ be the maximal destabilizing sub-line bundle of degree $e$. Let $\dev_{\mathcal{L}}:\HH\to \PP_{\CC}^1$ be the associated developing map.
    Then the first Lyapunov exponent associated to $\mathcal{V}$ is given by
	\[\lambda_1(\mathcal{V})=\frac{e}{g-1}+4\pi\lim_{r\to \infty}  \frac{1}{r}\int_{0}^{r}\frac{\sharp \{\dev^{-1}(x)\cap D_r(z)\}}{\vol(D_r(z))} \dd r\]
for almost any $z\in \HH$ and almost any $x\in \PP_{\CC}^1$. 
\end{prop}
\begin{proof}
The statement follows direcly from Theorem \ref{thm:conjcompact} and the  equivalent definition of bad locus given in Remark \ref{rem:geombad}.
\end{proof}

\begin{rem}
	The main result of  \cite{deroin} relating Lyapunov exponents to the covering degree defined in term of the developing map of a projective structure is a special case of the last theorem for $e=g-1$, meaning in the case of  the oper locus or equivalently of $\hol(\Proj(C))$. In \cite[Theorem F]{deroin}, the authors generalized the main result in the context of branched projective structures. Proposition \ref{prop:rank2shatz} is equivalent to their generalization. Indeed, as explained in \cite{biswasbranched}, there is an isomorphism between $(\mathcal{V}/\mathcal{L})\otimes \mathcal{L}^*$ and $T_C\otimes \mathcal{O}_C(S)$, where $S$ is the branching divisor of the branched projective structure associated to the sub-line bundle $\mathcal{L}\subset \mathcal{V}$. In particular then $\deg(\mathcal{L})=g-1-\deg(S)/2$.
	
	The difference in the equality of Proposition  \ref{prop:rank2shatz} and the ones in \cite{deroin} is given by the different normalization  of the hyperbolic metric on $C$.
\end{rem}

In the case of the maximal stratum, which is the oper locus $\Op_2(C)$, we can say something more using the description via holonomy of projective structures given by Proposition \ref{prop:operproj}. Recall that in remark \ref{rem:bers} we defined the subset $B(C)\subset \Op_2(C)$  given by the Bers embedding.

\begin{prop}
	\label{prop:operboundrank2}
	The Lyapunov exponent function restricted to the oper locus is greater or equal than 1  and unbounded. Moreover it is one if the representation belongs to  the closed subset $\overline{B(C)}\subset \Op_2(C)$ containing the uniformizing representation.
\end{prop}
\begin{proof}
 On the oper locus the maximal degree sub-line bundle is isomorphic to a square root $\mathcal{K}_C^{1/2}$ of the canonical bundle which has degree equal to $g-1$. Hence the Proposition \ref{thm:ekmz} gives that the top Lyapunov exponent has to be greater than $\lambda_1\geq 1$. Moreover the top Lyapunov exponent function is unbounded by Theorem \ref{thm:shatzunbounded}.

 Since the closed locus $\overline{B(C)}\subset \Op_2(C)$ is defined by the Bers embedding, by the density theorem  the standard developing map giving  the inclusion of the maximal sub-line bundle $\mathcal{V}_1\subset \mathcal{V}$ (see Proposition \ref{prop:operproj}) is a biholomorphism of $\HH$ onto one of the two domain of discontinuity of the associated representation. Hence the image of the developing map does not intersect the limit set of the representation. By the condition of Corollary \ref{cor:equalityrank2noncompact} we get equality.
\end{proof}

\begin{rem}
\label{rem:deroinstrict}
In \cite{deroin} they are able to prove the inverse of the second statement of the last proposition, namely that if the representation is in the complement $\Op_2(C)\setminus \overline{B(C)}$ of the Bers embedding, then $\lambda_1>1$. The tool that they can use is that they know that the support of the harmonic  measure is in the limit set and that the error term vanishes if and only if the image of the developing map intersect this support.
\end{rem}

\subsection{Lyapunov exponents for variations of Hodge structures}

Recall that the locus in $\drhamtwo$ corresponding to complex variation of Hodge structures is quite well understood.
We described its connected components $P_e$  in Section \ref{subsec:vhsrank2}. Since for $e=0$ we get weight zero variations of Hodge structures and for $0<e\leq g-1$ we get weight 1 variation of Hodge structures, we can compute exactly the associated top Lyapunov exponent in all of these loci.


\begin{prop}
	The top Lyapunov exponent function restricted on the connected component $P_e$ of the variation of Hodge structure locus is constant and given by
	\[{\lambda_1}_{|P_e}=\frac{e}{g-1},\quad e=0,\dots, g-1.\] 
\end{prop}
\begin{proof}
We already saw in Proposition \ref{prop:weight0} that the Lyapunov spectrum of a weight zero variation of Hodge structure is trivial. By the description of $P_e$ we know that if $\mathcal{H}\in P_e$, then the Hodge bundle $\mathcal{H}^{1,0}\subset \mathcal{H}$ has degree $e$ and the result follows from  Theorem \ref{thm:ekz} since we are dealing with weight 1 variation of Hodge structures.
\end{proof}

\section{Open questions about the top Lyapunov exponent  on Hitchin components}
\label{sec:unbounded}

Let $C$ be a compact Riemann surface. Recall   that Hitchin components are connected components of the real character variety $\Hom(\pi_1(C),\SL_n(\RR))//\SL_n(\RR)$ containing symmetric powers of Fuchsian representations. In particular the Hitchin component in rank $2$ is  Teichm\"uller space $\mathcal{T}(C)$.

There have been lately a focus on the study of dynamical invariants on Hitchin components. Recall for example the main result of \cite{poitriesambarino} that gives a bound of the critical exponent on Hitchin components that is attained if and only if the representation is the symmetric power of a Fuchsian representation. 
Unlike for critical exponent, we still cannot prove  a lower bound for the top Lyapunov exponent function on Hitchin components. We performed computer experiments computing the top Lyapunov exponent  on the Hitchin component of rank three  character variety $\Hom(\Delta(3,3,4),\SL_3(\RR))//\SL_3(\RR)$ associated to the triangle group $\Delta(3,3,4)$.  We were able to perform experiments thanks to the explicit description of matrices in the Hitchin components given in \cite{longreid}. The experiments indicated that  the top Lyapunov exponent function is unbounded and grows logarithmically near the boundary of the character variety and moreover that this function is greater than $2$ on this family, where $2$ here represents  the top Lyapunov exponent of the second symmetric power of the uniformizing representation. We believe that the top Lyapunov exponent can be related to other invariants like the critical exponent or the minimal area. Even though this relation is still unknown, if true one could use it together with the main result of \cite{poitriesambarino}  to prove a lower bound for the top Lyapunov exponent function on Hitchin components. Even tough such a bound has still to be proven and investigated, we want to state the following conjecture, which should be analogous to the bound of the critical exponent recalled above.   

\begin{conjecture}
	\label{conj:boundhitchin}
	The top Lyapunov exponent function on the $n$-th Hitchin component is unbounded and   greater or equal than $n-1$ , which is the top Lyapunov exponent of the $(n-1)$-th symmetric power of the uniformizing representation.
\end{conjecture}

It is not clear to expect if, as for the critical exponent, the attainment of the bound would imply that the representation is the symmetric power of the uniformizing one. 
The unboundedness on Teichmüller space should follow from \cite{dujfavre}. This would imply unboundedness on every Hitchin components. 

Notice that the conjecture is the analogous for Hitchin components of Proposition \ref{prop:lyapoper} and Theorem \ref{thm:shatzunbounded} which are about the oper locus.

\bibliographystyle{amsalpha}
\bibliography{reprvariety}
\end{document}